\documentclass[reqno]{amsart}
\usepackage{amsmath}
  \usepackage{graphics}
\usepackage{graphicx}

 \usepackage[colorlinks=true]{hyperref}

\usepackage{tikz}
\usepackage{pgfplots}

\hypersetup{urlcolor=blue, citecolor=red}

\newtheorem{theorem}{Theorem}[section]

\newtheorem{lemma}[theorem]{Lemma}
\newtheorem{proposition}{Proposition}

\theoremstyle{definition}
\newtheorem{definition}[theorem]{Definition}
\newtheorem{remark}{Remark}
\newtheorem*{notation}{Notation}

\DeclareMathOperator{\rep}{rep}

\DeclareMathOperator{\re}{Re}

\DeclareMathOperator{\sech}{sech}
\DeclareMathOperator{\Span}{span}
\DeclareMathOperator{\ess}{ess}
\DeclareMathOperator{\ran}{ran}
 \DeclareMathOperator{\csch}{csch}
\DeclareMathOperator{\eq}{eq}
 \DeclareMathOperator{\opl}{\emph{\textbf{l}}}
   \newcommand\dom{\operatorname{dom}}
 \newcommand{\bb}[1]{\mathbf{#1}}
\newcommand{\EE}{\mathcal E}

\title[Instability of excited states for the NLS-$\delta$ equation on a star graph] 
      {On the orbital instability of excited states for the  NLS equation with the $\delta$-interaction\\ on a star graph}

\author[Jaime Angulo Pava and Nataliia Goloshchapova]{}

\subjclass[2010]{Primary: 35Q55, 81Q35, 37K40, 37K45; Secondary: 47E05.}
 \keywords{Nonlinear Schr\"odinger equation, point interaction, self-adjoint extension, deficiency indices, orbital stability, standing wave, star graph, power nonlinearity, analytic perturbation.}

 \email{angulo@ime.usp.br}
 \email{nataliia@ime.usp.br}

\thanks{J. Angulo was supported partially by Grant CNPq/Brazil.  N. Goloshchapova  was supported by FAPESP under the project 2016/02060-9.}

\thanks{$^*$ Corresponding author: nataliia@ime.usp.br}

\begin{document}
\maketitle

\centerline{\scshape Jaime Angulo Pava and Nataliia Goloshchapova$^*$}
\medskip
{\footnotesize
 \centerline{Rua do Mat\~{a}o, 1010}
   \centerline{Cidade Universit\'{a}ria, S\~{a}o Paulo - SP, 05508-090, Brazil}

} 

%

\bigskip


\begin{abstract}
We study  the nonlinear Schr\"odinger equation (NLS) on a star graph $\mathcal{G}$. At the vertex an interaction occurs described by a boundary condition
of delta type with strength $\alpha\in \mathbb{R}$.
We investigate  the  orbital instability of the standing waves $e^{i\omega t}\mathbf{\Phi}(x)$ of the NLS-$\delta$ equation with attractive power  nonlinearity on $\mathcal{G}$ when the profile $\bb \Phi(x)$  has mixed structure (i.e. has  bumps and tails). In our approach  we essentially use the extension theory of symmetric operators by Krein - von Neumann, and the analytic perturbations theory, avoiding  the variational techniques  standard in the stability study.  We also prove the  orbital stability of the unique standing wave solution to the NLS-$\delta$ equation with repulsive nonlinearity.
\end{abstract}

\section{Introduction}
Let $\mathcal{G}$ be a star graph,  i.e. $N$
half-lines joined at the vertex $\nu=0$. On  $\mathcal{G}$ we consider the   following nonlinear Schr\"odinger equation
\begin{equation}\label{NLS_graph_ger}
i\partial_t \mathbf{U}(t,x)+\partial_x^2\mathbf{U}(t,x) +\mu|\mathbf{U}(t,x)|^{p-1}\mathbf{U}(t,x)=0,
\end{equation}
where $\mathbf{U}(t,x)=(u_j(t,x))_{j=1}^N:\mathbb{R}\times \mathbb{R}_+\rightarrow \mathbb{C}^N$,\, $\mu=\pm 1$, \, $p>1$, and nonlinearity acts componentwise,  i.e. $(|\mathbf{U}|^{p-1}\mathbf{U})_j=|u_j|^{p-1}u_j$.

Practically, equation  \eqref{NLS_graph_ger} means that on each edge of the graph, i.e. on each half-line, we have
$$i\partial_t u_j(t,x)+\partial_x^2u_j(t,x) +\mu|u_j(t,x)|^{p-1}u_j(t,x)=0,\,\,x>0,\,\,j\in\{1,...,N\}.$$

 A complete description of this model requires smoothness conditions along the edges and some junction conditions at the vertex $\nu=0$.  The family of self-adjoint conditions naturally arising  at the vertex $\nu=0$ of the star graph $\mathcal{G}$  has the following description
\begin{equation}\label{s-a_cond}
(U-I)\mathbf{U}(t,0)+i(U+I)\mathbf{U'}(t,0)=0,
\end{equation}
where $\mathbf{U}(t,0)=(u_j(t,0))_{j=1}^N$, $\mathbf{U'}(t,0)=(u'_j(t,0))_{j=1}^N$,  $U$ is  an arbitrary unitary $N\times N$ matrix, and $I$ is the $N\times N$ identity matrix.
Conditions \eqref{s-a_cond} at $\nu=0$ define the $N^2$-parametric family of self-adjoint extensions  of the closable  symmetric operator \cite[Chapter 17]{BlaExn08}
$$\mathbf{H}_0=\bigoplus\limits_{j=1}^N\frac{-d^2}{dx^2},\quad\dom(\mathbf{H}_0)=\bigoplus\limits_{j=1}^N C_0^\infty(\mathbb{R}_+).$$

In this paper we consider the  matrix $U$ which corresponds to the so-called $\delta$-interac\-tion  at vertex $\nu=0$. More precisely, the matrix
$$
U=\frac{2}{N+i\alpha}\mathcal{I}-I,\qquad \alpha\in \mathbb{R}\setminus\{0\},
$$
 where $\mathcal{I}$ is the $N\times N$ matrix whose all entries
equal 1, induces the following nonlinear Schr\"odinger equation with $\delta$-interaction (NLS-$\delta$) on the star graph $\mathcal{G}$
\begin{equation}\label{NLS_graph}
i\partial_t \mathbf{U}-\mathbf{H}^\alpha_\delta\mathbf{U} +\mu|\mathbf{U}|^{p-1}\mathbf{U}=0.
\end{equation}
Here $\mathbf{H}^\alpha_\delta$ is the self-adjoint operator on $L^2(\mathcal{G})$ defined  for $\mathbf{V}=(v_j)_{j=1}^N$ by
\begin{equation}\label{D_alpha}
\begin{split}
(\mathbf{H}^\alpha_\delta \mathbf{V})(x)&=\left(-v_j''(x)\right)_{j=1}^N,\quad x> 0,\\
D_\alpha:=\dom(\mathbf{H}^\alpha_\delta)&=\left\{\mathbf{V}\in H^2(\mathcal{G}): v_1(0)=...=v_N(0),\,\,\sum\limits_{j=1}^N  v_j'(0)=\alpha v_1(0)\right\}.
\end{split}
\end{equation}
 Condition at $\nu=0$ can be considered  as an analog of  $\delta$-interaction condition  for the Schr\"odinger operator on
the line  (see \cite{AlbGes05}), which justifies the  name of the equation. The case $\alpha<0$ refers to the presence of the potential well at the vertex, and $\alpha>0$ means the presence of a potential barrier.   When
$\alpha = 0$, one arrives at the known Kirchhoff
condition which corresponds to the free flow.



It is worth noting  that the quantum graphs (metric graphs equipped with  a linear  Hamiltonian $\bb H$) have been a very developed subject in the last couple of decades. They give  simplified models in mathematics, physics, chemistry, and engineering, when one considers propagation of waves of various types through a quasi one-dimensional (e.g. meso- or nanoscale) system that looks like a thin neighborhood of a  graph (see \cite{BK, Exn08, K, Mug15, Post12} for details and references).  In particular, a metric graph appears as  the natural limit of thin tubular structure, when the radius of a tubular structure tends  to zero \cite{Post12}.

The nonlinear PDEs on graphs have been studied in the last ten years in the context of existence, stability, and propagation of solitary waves. For instance, in \cite{Noj14} the author provides an overview of some recent results and open problems for  NLS  on graphs. The analysis of the behavior of the NLS  equation on networks is currently growing subject due to its relative analytical  simplicity (the metric graph is essentially one-dimensional) and  various physical applications involving wave propagation in graph-like structures (see the references in  \cite{ CacFin17, Kai17, Noj14}). In particular, two main fields where NLS appears as a preferred model are nonlinear optics and Bose-Einstein condensates.

The main purpose  of this work is the investigation of the  stability properties of the standing wave solutions
$$
\mathbf{U}(t,x)=e^{i\omega t}\mathbf{\Phi}(x)=\left(e^{i\omega t}\varphi_{j}(x)\right)_{j=1}^N
$$
to NLS-$\delta$ equation   \eqref{NLS_graph}.
In a series of papers    Adami,  Cacciapuoti,  Finco, and Noja (see \cite{AdaNoj15} and references therein) investigated variational and stability  properties of standing wave solutions  to equation   \eqref{NLS_graph} for $\mu=1$ (\textit{attractive nonlinearity}). In \cite{AdaNoj14}  it was shown that all possible profiles $\bb \Phi(x)$ belong to the specific family of $\left[\frac{N-1}{2}\right]+1$  vector functions (see Theorem \ref{1bump} below) consisting of bumps and tails.    It was  proved that there exists a global minimizer of the constrained  NLS action  for  $-N\sqrt{\omega}<\alpha<\alpha^*<0$. This minimizer coincides with the $N$-tails stationary state symmetric under permutation of edges, which consists of decaying tails (notice also that this profile  minimizes NLS energy under fixed mass constraint for sufficiently small mass \cite{AdaNoj14a}).

 Using minimization property,  the authors proved the orbital  stability of this  $N$-tails stationary state  in the case  $-N\sqrt{\omega}<\alpha<\alpha^*<0$.

In \cite{AdaNoj15} it was shown that although the constrained minimization problem does not admit global minimizers for large mass, the $N$-tails stationary state is still a local minimizer of the constrained energy which induces the 
 orbital stability for any $-N\sqrt{\omega}<\alpha<0$.
The orbital stability of $N$-tails (bumps) profile was studied in   \cite{AngGol16} in the framework of the extension theory. In particular, it was proved that $N$-bumps profile $\mathbf{\Phi}^\alpha_0$ (for $\alpha>0$) is   orbitally unstable in  $\mathcal{E}$ for $1<p\leq 3,\, \omega>\tfrac{\alpha^2}{N^2}$,  and   $3<p<5,\, \omega>\omega_0> \tfrac{\alpha^2}{N^2}$ (see Theorem 1.1 in \cite{AngGol16}). Moreover, in \cite{AngGol16}  we considered the NLS equation with $\delta'$-interaction.

In the case $\alpha<0$ it was shown in \cite{AdaNoj14} that the NLS action functional grows when the number of tails in the stationary state increases, i.e. one can call the rest of the profiles (except $N$-tails stationary state)  \textit{excited stationary states } (see Subsection \ref{exist}).
 This is a subject of special interest because there are only few cases where excited states of NLS equations are explicitly known.

In the present paper we provide sufficient condition for the orbital instability of the excited states of \eqref{NLS_graph}.   Moreover, we  obtain the novel result on the orbital stability/instability of the standing waves in the case $\alpha>0$.
 \begin{theorem}\label{main}
 Let $\alpha \neq 0$, \,$\mu=1$,\, $k\in\left\{1,...,\left[\tfrac{N-1}{2}\right]\right\}$, and  $ \omega>\tfrac{\alpha^2}{(N-2k)^2}$. Let also  the profile $\mathbf{\Phi}_{k}^\alpha$ be  defined by   \eqref{Phi_k}, and the spaces $\EE$, $\EE_k$ be defined in notation section.
  Then the following assertions hold.
 \begin{enumerate}
\item[(i)] Let $\alpha<0$, then
  \begin{enumerate}
\item[1)] for  $1<p\leq 5$ the standing wave $e^{i\omega t}\mathbf{\Phi}_{k}^\alpha$ is orbitally unstable in $\mathcal{E}$;
\item[2)] for $p>5$ there exists  $\omega_k^*>\tfrac{\alpha^2}{(N-2k)^2}$ such that  the standing wave $e^{i\omega t}\mathbf{\Phi}_{k}^\alpha$ is orbitally unstable in $\mathcal{E}$ as $\omega\in (\tfrac{\alpha^2}{(N-2k)^2}, \omega_k^*)$.
\end{enumerate}
\item[(ii)] Let $\alpha>0$, then
\begin{enumerate}
\item[1)] for $1<p\leq 3$ the standing wave $e^{i\omega t}\mathbf{\Phi}_{k}^\alpha$ is orbitally stable in $\mathcal{E}_k$; 
\item[2)] for $3<p<5$  there exists $\hat\omega_k>\tfrac{\alpha^2}{(N-2k)^2}$ such that  the standing wave $e^{i\omega t}\mathbf{\Phi}_{k}^\alpha$ is orbitally unstable in $\mathcal{E}$ as $\omega\in (\tfrac{\alpha^2}{(N-2k)^2}, \hat\omega_k)$, and  $e^{i\omega t}\mathbf{\Phi}_{k}^\alpha$ is orbitally stable in $\mathcal{E}_k$ as $\omega\in (\hat\omega_k,\infty)$; 
\item[3)] for $p\geq 5$ the standing wave $e^{i\omega t}\mathbf{\Phi}_{k}^\alpha$ is orbitally unstable in $\mathcal{E}$.
\end{enumerate}
\end{enumerate}
\end{theorem}
Recently similar results were obtained in \cite[Theorem 3.2]{Kai17}. In particular, the author proved the spectral instability of  $\mathbf{\Phi}_{k}^\alpha$ in the cases $\alpha<0$,\, $k\geq 1$ and $\alpha>0,\, k\geq 0$. His method essentially uses the generalization of the Sturm theory for
the Schr\"odinger operators on the star graph.

 As it was noted above, the Kirchhoff  condition on $\mathcal{G}$ corresponds to $\alpha=0$ in \eqref{D_alpha}.
 In \cite[Theorem 5]{AdaNoj14} it was shown that  for $N$ even  there exists one-parametric family of soliton profiles  given by $\bb \Phi_a(x)=(\varphi_{a,j}(x))_{j=1}^N,$ where
 \begin{equation*}
 \varphi_{a,j}(x)= \left\{
                    \begin{array}{ll}
                      \varphi_0(x-a), & \quad\hbox{$j=1,...,N/2$;} \\
                       \varphi_0(x+a), & \quad\hbox{$j=N/2+1,...,N,$}
                    \end{array}
                  \right.\quad a\in\mathbb{R},
 \end{equation*}
 while for $N$   odd the unique profile is given by $\bb \Phi_0(x)=(\varphi_0(x))_{j=1}^N$, with $\varphi_0(x)$ defined by \eqref{var_0}. In particular, $\bb \Phi_0(x)$ is the stationary state  for any $N\geq 2$.
In \cite{KaiPel17a} the authors considered the spectral instability of the family $\bb \Phi_a(x)$ for more general setting (for the generalized Kirchhoff condition), while in \cite{KaiPel17b} they studied orbital instability of $\bb \Phi_0(x)$. Namely, the authors proved in \cite[Theorem 2.6]{KaiPel17b}  that  for $2\leq p< 5$ the standing wave $e^{i\omega t}\bb \Phi_0(x)$  is orbitally unstable in $\EE$.
We  complement this result by the following theorem.
 \begin{theorem}\label{Kirsh}
Let $\omega>0$, then
\begin{itemize}
\item[$(i)$] for $1<p<5$  the standing wave $e^{i\omega t}\bb \Phi_0(x)$ is orbitally stable in $\EE_{\eq}$;
\item[$(ii)$]  for $p>5$  the standing wave $e^{i\omega t}\bb \Phi_0(x)$ is orbitally unstable in $\EE$.
\end{itemize}
\end{theorem}

\noindent The instability part was  announced  in \cite[Remark 2.8]{KaiPel17b} without proof.

In Section \ref{repuls} we  consider  model \eqref{NLS_graph} with $\mu=-1$ (\textit{repulsive nonlinearity}).  We prove the following new result on the orbital stability of the unique ($N$-tails) stationary state  $\bb \Phi_\alpha=(\varphi_{\alpha})_{j=1}^N$, where
 \begin{equation}\label{1varphi_rep}
 \varphi_\alpha(x)=\left[\frac{(p+1)\omega}{2} \csch^2\left(\frac{(p-1)\sqrt{\omega}}{2}x+\coth^{-1}\left(\frac{-\alpha}{N\sqrt{\omega}}\right)\right)\right]^{\frac{1}{p-1}},\qquad x>0,
 \end{equation}
with $\alpha<0$  and $0<\omega< \tfrac{\alpha^2}{N^2}$. More exactly, we prove
  \begin{theorem}\label{main_rep}
 Let $\alpha<0$,\, $0<\omega< \tfrac{\alpha^2}{N^2}$, and $\bb \Phi_\alpha$ be defined by \eqref{1varphi_rep}. Then the standing wave  $e^{i\omega t}\bb \Phi_\alpha$ is orbitally stable in $\EE$.
 \end{theorem}

 Our approach contains new original technique. It  does not use variational analysis, and it is based on  the extension theory of symmetric operators, the  analytic perturbations theory,   Grillakis-Shatah-Strauss and Ohta approach (see \cite{GrilSha87, GrilSha90, oh}).


\begin{notation}
Let $A$ be a densely defined symmetric operator in a Hilbert space $\mathcal{H}$. The domain of $A$ is denoted by $\dom(A)$. The \textit{deficiency subspaces} and the \textit{deficiency  numbers} of $A$ are denoted by $\mathcal{N}_{\pm}(A):=\ker(A^*\mp iI)$ and  $n_\pm(A):=\dim \mathcal{N}_{\pm}(A)$ respectively. The number of negative eigenvalues counting multiplicities is denoted by  $n(A)$ (\textit{the Morse index}). The spectrum and the resolvent set of $A$ are denoted by $\sigma(A)$ and $\rho(A)$.

  We denote by  $\mathcal{G}$ the star graph constituted by $N$ half-lines attached to a common vertex $\nu=0$. On the graph we define
  \begin{equation*}
  L^p(\mathcal{G})=\bigoplus\limits_{j=1}^NL^p(\mathbb{R}_+),\,p>1,\quad H^1(\mathcal{G})=\bigoplus\limits_{j=1}^NH^1(\mathbb{R}_+),\quad H^2(\mathcal{G})=\bigoplus\limits_{j=1}^NH^2(\mathbb{R}_+).
\end{equation*}
For instance, the norm in $L^p(\mathcal{G})$ is defined by $$||\bb V||^p_{L^p(\mathcal{G})}=\sum\limits_{j=1}^N||v_j||^p_{L^p(\mathbb{R}_+)},\quad \mathbf{V}=(v_j)_{j=1}^N.$$
 By  $||\cdot||_p$ we denote  the norm in  $L^p(\mathcal{G})$, and  $(\cdot,\cdot)_2$   denotes the scalar product in  $L^2(\mathcal{G})$.

We also denote by $\EE$ and $L_k^2(\mathcal{G})$ the spaces
\begin{equation*}
\begin{split}
&\EE=\{\mathbf{V}=(v_j)_{j=1}^N\in H^1(\mathcal{G}): v_1(0)=...=v_N(0)\},
\\
&L_k^2(\mathcal{G})=\left\{\begin{array}{c}\bb V=(v_j)_{j=1}^N\in L^2(\mathcal{G}): v_1(x)=...=v_k(x),\\
v_{k+1}(x)=...=v_N(x),\, x\in\mathbb{R}_+
\end{array}\right\},
\end{split}
\end{equation*}
 and $\EE_k=\EE\cap L_k^2(\mathcal{G})$.
Finally,
 $$L^2_{\eq}(\mathcal{G})=\{\mathbf{V}=(v_j)_{j=1}^N\in L^2(\mathcal{G}): v_1(x)=...=v_{N}(x),\, x\in\mathbb{R}_+\},$$
 and $\EE_{\eq}=\EE\cap L^2_{\eq}(\mathcal{G})$.
 \end{notation}

\section{Preliminaries}\label{prelim}
\subsection{Well-posedness}
The well posedness is the crucial assumption in the stability theory. In \cite{AdaNoj14} the problem of well-posedness of the NLS-$\delta$ equation has been studied in the case $\mu=1$. Recently  we have completed and extended  mentioned result (see \cite{AngGol16}). Below we recall our results, as well as generalize them on the case $\mu=-1$.

In \cite{AngGol16} (see Lemmas 3.1 and 3.3, and the proof of Theorem 3.4) we prove the following three technical lemmas.
\begin{lemma}\label{rela00} Let  $\{e^{-it\bb H_{\delta}^\alpha}\}_{t\in \mathbb R}$ be  the family of unitary operators  associated to  NLS-$\delta$ model \eqref{NLS_graph}. Then for every $ \bb V=(v_j)_{i=1}^N\in \EE$ we have
 \begin{equation*}\label{group_comut0}
\partial_x(e^{-it\bb H_{\delta}^\alpha }\bb V)=-e^{-it\bb H_{\delta}^\alpha}\bb V' +
\mathcal B(\bb V'),
\end{equation*}
where $\mathcal B(\bb V')=(2e^{it\partial^2_x}\tilde{v}_j)_{j=1}^N$, with $ \tilde{v}_j(x)=\left\{\begin{array}{c}
v'_j(x),\,\ x\geq 0,\\
0,\quad x<0
\end{array}\right.$, and $e^{it\partial^2_x}$ is the unitary group associated with the free Schr\"odinger operator on $\mathbb{R}$.
\end{lemma}
The proof is based on the following representation of the group $e^{-it\bb H_\delta^\alpha}$ for $\alpha>0$ (for the case $\alpha<0$ see \cite[Remark 3.2]{AngGol16})
\begin{equation}\label{group0}
  e^{-it\bb H_\delta^\alpha}\bb V(x)=\tfrac{i}{\pi}\int\limits_{-\infty}^\infty e^{-it\tau^2}\tau \bb R_{i\tau}\bb V(x)d\tau,
\end{equation}
where $\bb R_{z}\bb V=(\bb H_\delta^\alpha+z^2 I)^{-1}\bb V$ has the components
 \begin{equation}\label{res0}
(\bb R_{z}\bb V)_j(x)=\tilde{c}_je^{-z x}+\frac{1}{2z}\int\limits_0^\infty v_j(y)e^{-|x-y|z}dy.
 \end{equation}
 The coefficients $\tilde{c}_j$ are determined by  the system
 \begin{equation}\label{c_tilde_delta}
\left(\begin{array}{ccccc}
  1&-1&0&...&0\\
  0&1&-1&...&0\\
    \vdots &\vdots&\vdots&  &\vdots\\
  0&0&0&...&-1\\
  \tfrac{\alpha}{N}+z&\tfrac{\alpha}{N}+z&\tfrac{\alpha}{N}+z&...&\tfrac{\alpha}{N}+z
  \end{array}\right)\left(\begin{array}{ccccc}\tilde{c}_1\\
  \\
 \vdots \\
  \\
  \tilde{c}_N
  \end{array}
  \right)=-\frac{1}{z}\left(\begin{array}{ccccc}t_1(z)-t_2(z)\\
  \vdots\\
  \\
 t_{N-1}(z)-t_N(z) \\
  (\frac{\alpha}{N}-z)\sum\limits_{j=1}^Nt_j(z) \end{array}
  \right),
\end{equation}
where  $ t_j(z)=\frac12 \int\limits_0^\infty v_j(y)e^{-z y}dy$.
Another  two lemmas follow from  formulas \eqref{group0}-\eqref{c_tilde_delta}.
\begin{lemma}\label{persistence} The family of unitary operators $\{e^{-it\bb H^{\alpha}_\delta}\}_{t\in \mathbb R}$  preserves the space $\mathcal E$, i.e. for $\bb U_0\in \mathcal E$ we have $e^{-it\bb H_{\delta}^\alpha} \bb U_0\in  \mathcal E$.
\end{lemma}
\begin{lemma}\label{persistence_k} The  family of unitary operators $\{e^{-it\bb H^{\alpha}_\delta}\}_{t\in \mathbb R}$  preserves the space $\mathcal E_k$, i.e. for $\bb U_0\in \mathcal E_k$ we have $e^{-it\bb H_{\delta}^\alpha} \bb U_0\in  \mathcal E_k$.
\end{lemma}

Using essentially the above lemmas we prove the following extended well-posed\-ness result.
\begin{theorem}\label{loc_well_posed}
 Let $p>1$, $\mu=\pm 1$. For any $\bb U_0 \in \EE$, there exists $T > 0$ such that
equation \eqref{NLS_graph} has a unique solution $\bb U(t) \in C ([-T,T],
\EE)
\cap C^1 ([-T,T], \EE')$ satisfying $\bb U(0)=\bb U_0$.  In particular, if  $\bb U_0 \in \EE_k$, then $\bb U(t)\in \EE_k$. For
each $T_0\in (0, T)$ and $p>2$, the mapping 
$
\bb U_0\in \mathcal E \to \bb U \in C ([-T_0,T_0],
\mathcal
E)
$
is of class $C^2$.

Furthermore, the conservation of energy and mass holds, that is,
$$E(\bb U(t))=E(\bb U_0),\quad M(\bb U(t))=||\bb U(t)||_2^2=||\bb U_0||_2^2,$$
where the energy $E$  is defined by
\begin{equation}\label{energy}
E(\bb U)=\tfrac 1{2}||\bb U'||_2^2-\tfrac {\mu}{p+1}||\mathbf{U}||_{p+1}^{p+1}+\tfrac{\alpha}{2}|u_1(0)|^2, \quad \mathbf{U}=(u_j)_{j=1}^N\in\EE.
\end{equation}
 \end{theorem}
 \begin{proof} The proof of theorem repeats the one of \cite[Theorem 3.4]{AngGol16}. We present it to  give a self-contained exposition of the subject. The local well-posedness result in $\EE$ follows from standard arguments of the Banach fixed point theorem applied to nonlinear Schr\"odinger equations (see  \cite{Caz}).
 Consider the  mapping $J_{\bb U_0}: C([-T, T], \EE )\longrightarrow C([-T,T], \EE)$ given by
\begin{equation}\label{J_U}
J_{\bb U_0}[\bb U](t)=e^{-it\bb H_{\delta}^\alpha}\bb U_0+\mu i\int\limits_0^te^{-i(t-s)\bb H_{\delta}^\alpha}|\bb U(s)|^{p-1}\bb U(s)ds,
\end{equation}
where $e^{-it\bb H_{\delta}^\alpha}$ represents the unitary group associated to model \eqref{NLS_graph}.
One needs to show that  the map $J_{\bb U_0}$ is well-defined. We start by estimating the nonlinear term $|\bb U(s)|^{p-1}\bb U(s)$. Using  the one-dimensional Gagliardo-Nirenberg inequality one can show (see formula (2.3) in \cite{AdaNoj14})
\begin{equation}\label{G-N_graph0}
||\bb U||_q \leq C||\bb U'||_2^{\frac12-\frac1q}||\bb U||_2^{\frac12+\frac1q},\quad q>2,\, C>0.
\end{equation}
Using \eqref{G-N_graph0}, the relation $|(|f|^{p-1} f)'|\leq C_0 |f|^{p-1} |f'|$ and   H\"older's inequality, we obtain  for $\bb U\in H^1(\mathcal{G})$
\begin{equation}\label{pres_space0}
|||\bb U|^{p-1}\bb U||_{H^1(\mathcal{G})}\leq C_1||\bb U||^p_{H^1(\mathcal{G})}.
\end{equation}
Let $\bb U_0, \bb U \in \EE$, then  from \eqref{J_U},  inequality \eqref{pres_space0}, $L^2$-unitarity of $e^{-it\bb H_\delta^\alpha}$ and   $e^{it\partial^2_x}$, and Lemma \ref{rela00}, we obtain the  estimate
\begin{equation*}\label{neq_well_1}
||J_{\bb U_0}[\bb U](t)||_{H^1(\mathcal{G})}\leq C_2||\bb U_0||_{H^1(\mathcal{G})}+C_3T\sup\limits_{s\in[0,T]}||\bb U(s)||^p_{H^1(\mathcal{G})},
\end{equation*}
where the positive constants $C_2, C_3$ do not depend on $\bb U_0$. Moreover, from Lemma \ref{persistence} we get $J_{\bb U_0}[\bb U](t)\in \EE$ for all $t$.

The  continuity  and the contraction property of $J_{\bb U_0}[\bb U](t)$ are proved in a  standard way.
Therefore, we obtain the existence of a unique solution for the Cauchy problem associated to  \eqref{NLS_graph}  on $\EE$.

Next, we recall that the argument based on the contraction mapping principle above has the advantage that if the nonlinearity $F(\bb U, \overline{\bb U})=|\bb U|^{p-1}\bb U$ has a specific regularity, then it is inherited by the mapping data-solution. In particular, following the ideas in the proof of  \cite[Corollary 5.6]{LinPon09}, we consider  for $(\bb V_0, \bb V)\in B(\bb U_0;\epsilon)\times C([-T, T], \mathcal E) $ the mapping
$$
\Gamma(\bb V_0, \bb V)(t)=\bb V(t)- J_{\bb V_0}[\bb V](t),\qquad t\in [-T, T].
$$
Then $\Gamma(\bb U_0, \bb U)(t)=0$ for all $t\in [-T, T]$. For $p-1$ being an even integer,  $F(\bb U, \overline{\bb U})$ is smooth, and therefore $\Gamma$ is smooth. Hence, using the arguments applied for  obtaining the local well-posedness in  $\EE$ above,  we can show that the operator $\partial_\bb V\Gamma(\bb U_0, \bb U)$ is one-to-one and onto. Thus, by the Implicit Function Theorem  there exists a smooth mapping $\bb \Lambda: B(\bb U_0;\delta)\to C([-T, T], \EE)$ such that $\Gamma(\bb V_0, \bb \Lambda (\bb V_0))=0$ for all $\bb V_0\in B(\bb U_0;\delta)$. This argument establishes the smoothness  property of the mapping data-solution associated to equation  \eqref{NLS_graph} when  $p-1$ is an even integer.

If  $p-1$ is  not  an even integer and $p>2$, then $F(\bb U, \overline{\bb U})$ is $C^{[p]}$-function,   and consequently the mapping data-solution is of  class $C^{[p]}$ (see \cite[Remark 5.7]{LinPon09}). Therefore, for $p>2$ we conclude that the mapping data-solution is at least of class $C^{2}$.

Finally, from the  uniqueness of the solution to the Cauchy problem for \eqref{NLS_graph} in $\EE$ and Lemma \ref{persistence_k}   we get that for $\bb U_0\in \EE_k$ the solution $\bb U(t)$ to the Cauchy problem belongs to $\EE_k$ for any $t$.

The proof of conservation laws repeats the one in \cite{Caz} (see Theorem 3.3.1, 3.3.5 and 3.3.9).
\end{proof}

\begin{lemma}\label{glob_well}
The local solution of the Cauchy problem for  equation \eqref{NLS_graph}  is extended globally for $p\in (1,5)$  in the case  $\mu=1$, and for $p>1$ in the case $\mu=-1$ (i.e. $T=+\infty$).
\end{lemma}
\begin{proof}
The case $\mu=1$  was considered  in \cite[Corollary 2.1]{AdaNoj14}. For $\mu=-1$ the trivial inequality
$$
\tfrac 1{2}||\bb U'||_2^2+\tfrac{\alpha}{2}|u_1(0)|^2= E(\bb U) + \tfrac {\mu}{p+1}||\mathbf{U}||_{p+1}^{p+1}< E(\bb U)
$$
induces the global existence for any $p>1$.
\end{proof}
\subsection{Existence  of standing waves}\label{exist}

Let us discuss briefly the existence of the standing wave solutions $\bb U(t,x)=e^{i\omega t}\bb \Phi(x)$
 to \eqref{NLS_graph}.  It is easily seen that the  amplitude $\mathbf{\Phi}\in D_\alpha$  satisfies the following  stationary equation
\begin{equation}\label{H_alpha}
\mathbf{H}^\alpha_\delta\mathbf{\Phi}+\omega\mathbf{\Phi}-\mu|\mathbf{\Phi}|^{p-1}\mathbf{\Phi}=0.
\end{equation}

 In \cite{AdaNoj14} the authors obtained the following description of all solutions   to equation \eqref{H_alpha} in the case $\mu=1$.
\begin{theorem}\label{1bump}
Let  $[s]$ denote the integer part of $s\in\mathbb{R}$, $\alpha\neq 0$ and $\mu=1$.  Then equation \eqref{H_alpha} has $\left[\tfrac{N-1}{2}\right]+1$ (up to permutations of the edges of $\mathcal{G}$) vector solutions $\mathbf{\Phi}_k^\alpha=(\varphi^\alpha_{k,j})_{j=1}^N, \,\,k=0,...,\left[\tfrac{N-1}{2}\right]$, which are given by
\begin{equation}\label{Phi_k}
\begin{split}
 \varphi_{k,j}^\alpha(x)&= \left\{
                    \begin{array}{ll}
                      \Big[\frac{(p+1)\omega}{2} \sech^2\Big(\frac{(p-1)\sqrt{\omega}}{2}x-a_k\Big)\Big]^{\frac{1}{p-1}}, & \quad\hbox{$j=1,...,k$;} \\
                     \Big[\frac{(p+1)\omega}{2} \sech^2\Big(\frac{(p-1)\sqrt{\omega}}{2}x+a_k\Big)\Big]^{\frac{1}{p-1}}, & \quad\hbox{$j=k+1,...,N,$}
                    \end{array}
                  \right.\\
                  \text{where}\;\; a_k&=\tanh^{-1}\left(\frac{\alpha}{(2k-N)\sqrt{\omega}}\right),\,\,\text{and}\,\,\,\,\omega>\tfrac{\alpha^2}{(N-2k)^2}.
                  \end{split}
\end{equation}
\end{theorem}

\begin{remark}\label{tails_bumps}
\begin{itemize}
\item[$(i)$]
Note that in the case $\alpha<0$ vector  $\mathbf{\Phi}_k^\alpha=(\varphi^\alpha_{k,j})_{j=1}^N$ has $k$ \textit{bumps} and $N-k$ \textit{tails}. It is easily seen that
 $\mathbf{\Phi}^{\alpha}_0$ is the \textit{N-tails profile}. Moreover, the $N$-tails profile is the only symmetric (i.e. invariant under permutations of the edges) solution of equation \eqref{H_alpha}.  In the case $N=5$ we have three types of profiles:  \textit{5-tails profile},  \textit{4-tails/1-bump profile} and  \textit{3-tails/2-bumps profile}. They are demonstrated on Figure 1 (from the left to the right).
 \item[$(ii)$]  In the case $\alpha>0$ vector  $\mathbf{\Phi}_k^\alpha=(\varphi^\alpha_{k,j})_{j=1}^N$ has $k$ \textit{tails} and $N-k$ \textit{bumps} respectively. For $N=5$ we have: \textit{5-bumps profile, 4-bumps/1-tail profile, 3-bumps/ 2-tails profile}. They are demonstrated on Figure 2 (from the left to the right).
 \end{itemize}
 \end{remark}
 	\begin{tikzpicture}[scale=0.6]
	\clip (-2,-3) rectangle (18,2);
	    \draw[-,color=gray] (2,1).. controls +(-0.7,-0.7) ..  (0.3,0.2);
       \draw[-,color=gray] (2,1).. controls +(-0.3,-0.9)  ..  (0.5,-1);
        \draw[-,color=gray] (2,1).. controls +(0.2,-1)  ..  (2.8,-1.2);
        \draw[-,color=gray] (2,1) .. controls +(0.2,-0.5)  ..  (4.2,-0.2);
            \draw[-,color=gray] (2,1).. controls +(0.6,-0.3)  ..  (3.5,1);
		\draw[-latex,thin](2,0)--++(-2,0);
        \draw[-latex, thin](2,0)--++(-1.8,-1.3);
        \draw[-latex, thin](2,0)--++(1,-1.7);
		\draw[-latex, thin](2,0)--++(2.5,-0.5);
        \draw[-latex, thin](2,0)--++(2,1);

		\begin{scope}[shift={(6,0)}]
	     \draw[-,color=gray] (2,1).. controls +(-0.5,0.5) and +(1.5,-0.1) ..  (0.3,0.2);
       \draw[-,color=gray] (2,1).. controls +(-0.3,-0.9)  ..  (0.5,-1);
        \draw[-,color=gray] (2,1).. controls +(0.2,-1)  ..  (2.8,-1.2);
        \draw[-,color=gray] (2,1) .. controls +(0.2,-0.5)  ..  (4.2,-0.2);
            \draw[-,color=gray] (2,1).. controls +(0.6,-0.3)  ..  (3.5,1);
		\draw[-latex,thin](2,0)--++(-2,0);
        \draw[-latex, thin](2,0)--++(-1.8,-1.3);
        \draw[-latex, thin](2,0)--++(1,-1.7);
		\draw[-latex, thin](2,0)--++(2.5,-0.5);
        \draw[-latex, thin](2,0)--++(2,1);
		\node[label={[xshift=1.3cm, yshift=-2cm] Figure 1}]{};	\end{scope}

		\begin{scope}[shift={(12,0)}]
	     \draw[-,color=gray] (2,1).. controls +(-0.5,0.5) and +(1.5,-0.1) ..  (0.3,0.2);
       \draw[-,color=gray] (2,1).. controls +(-0.3,-0.9)  ..  (0.5,-1);
        \draw[-,color=gray] (2,1).. controls +(0.2,-1)  ..  (2.8,-1.2);
        \draw[-,color=gray] (2,1) .. controls +(0.2,-0.5)  ..  (4.2,-0.2);
            \draw[-,color=gray] (2,1).. controls +(0.3,0.5) and +(-0.8,-0.5)  ..  (3.5,1);
		\draw[-latex,thin](2,0)--++(-2,0);
        \draw[-latex, thin](2,0)--++(-1.8,-1.3);
        \draw[-latex, thin](2,0)--++(1,-1.7);
		\draw[-latex, thin](2,0)--++(2.5,-0.5);
        \draw[-latex, thin](2,0)--++(2,1);
	\end{scope}
\end{tikzpicture}

\

\
\begin{tikzpicture}[scale=0.6]
	\clip (-2,-3) rectangle (18,2);
	    \draw[-,color=gray] (2,1).. controls +(-0.5,0.5) and +(1.5,-0.1) ..  (0.3,0.2);
       \draw[-,color=gray] (2,1).. controls +(-0.3,0.4) and +(1, 0.5) ..  (0.5,-1);
        \draw[-,color=gray] (2,1).. controls +(0.2,0.3) and +(-0.1, -0.1) ..  (2.8,-1.2);
        \draw[-,color=gray] (2,1) .. controls +(0.2,0.5) and +(-1.1, 0.3) ..  (4.2,-0.2);
            \draw[-,color=gray] (2,1)..  controls +(0.3,0.5) and +(-0.8,-0.5)  ..  (3.5,1);
		\draw[-latex,thin](2,0)--++(-2,0);
        \draw[-latex, thin](2,0)--++(-1.8,-1.3);
        \draw[-latex, thin](2,0)--++(1,-1.7);
		\draw[-latex, thin](2,0)--++(2.5,-0.5);
        \draw[-latex, thin](2,0)--++(2,1);

		\begin{scope}[shift={(6,0)}]
	     \draw[-,color=gray] (2,1).. controls +(-0.7,-0.7) ..  (0.3,0.2);
       \draw[-,color=gray] (2,1).. controls +(-0.3,0.4) and +(1, 0.5) ..  (0.5,-1);
        \draw[-,color=gray] (2,1).. controls +(0.2,0.3) and +(-0.1, -0.1) ..  (2.8,-1.2);
        \draw[-,color=gray] (2,1) .. controls +(0.2,0.5) and +(-1.1, 0.3) ..  (4.2,-0.2);
           \draw[-,color=gray] (2,1)..  controls +(0.3,0.5) and +(-0.8,-0.5)  ..  (3.5,1);
		\draw[-latex,thin](2,0)--++(-2,0);
        \draw[-latex, thin](2,0)--++(-1.8,-1.3);
        \draw[-latex, thin](2,0)--++(1,-1.7);
		\draw[-latex, thin](2,0)--++(2.5,-0.5);
        \draw[-latex, thin](2,0)--++(2,1);
		\node[label={[xshift=1.3cm, yshift=-2cm] Figure 2}]{};	\end{scope}

		\begin{scope}[shift={(12,0)}]
	    \draw[-,color=gray] (2,1).. controls +(-0.7,-0.7) ..  (0.3,0.2);
       \draw[-,color=gray] (2,1).. controls +(-0.3,0.4) and +(1, 0.5) ..  (0.5,-1);
        \draw[-,color=gray] (2,1).. controls +(0.2,0.3) and +(-0.1, -0.1) ..  (2.8,-1.2);
        \draw[-,color=gray] (2,1) .. controls +(0.2,0.5) and +(-1.1, 0.3) ..  (4.2,-0.2);
          \draw[-,color=gray] (2,1).. controls +(0.6,-0.3)  ..  (3.5,1);
		\draw[-latex,thin](2,0)--++(-2,0);
        \draw[-latex, thin](2,0)--++(-1.8,-1.3);
        \draw[-latex, thin](2,0)--++(1,-1.7);
		\draw[-latex, thin](2,0)--++(2.5,-0.5);
        \draw[-latex, thin](2,0)--++(2,1);
	\end{scope}
\end{tikzpicture}
\

\
 	In \cite{AdaNoj14} it was shown that for any $p>1$ there is  $\alpha^*<0$ such that for $-N\sqrt{\omega}<\alpha<\alpha^*$ the $N$-tails profile  $\mathbf{\Phi}^\alpha_{0}$ minimizes the action functional
 \begin{equation}\label{S_graph}
 S(\mathbf{V})=\tfrac 1{2}||\mathbf{V}'||_2^2+\tfrac{\omega}{2}||\mathbf{V}||_2^2 -\tfrac 1{p+1}||\mathbf{V}||_{p+1}^{p+1}+\tfrac{\alpha}{2}|v_1(0)|^2, \quad \mathbf{V}=(v_j)_{j=1}^N\in\EE,
 \end{equation}
 on the  Nehari manifold
 $$
 \mathcal N=\{\mathbf{V}\in \EE\setminus\{0\}: ||\mathbf{V}'||_2^2+\omega ||\mathbf{V}||_2^2 - ||\mathbf{V}||_{p+1}^{p+1}+\alpha |v_1(0)|^2=0\}.
 $$
 Namely, the $N$-tails profile $\mathbf{\Phi}^\alpha_{0}$ is the ground state  for the action $S$ on the manifold $\mathcal N$.
In \cite{AdaNoj15} the authors showed that $\bb \Phi_0^\alpha$ is a local minimizer of the energy functional  $E$ defined by \eqref{energy} among functions with equal  mass.

Note that  $\mathbf{\Phi}^\alpha_{k}\in \mathcal{N}$ for any $k$. In \cite{AdaNoj14} it was proved that for $k\neq 0$ and $\alpha<0$ we have   $S(\mathbf{\Phi}^\alpha_{0})<S(\mathbf{\Phi}^\alpha_{k})<S(\mathbf{\Phi}^\alpha_{k+1})$.
  This fact justifies the name \textit{excited states} for the  stationary states $\mathbf{\Phi}^\alpha_k,\, k\neq 0$. It is worth noting that the profiles $\mathbf{\Phi}^\alpha_{k},\, k\neq 0,$ are excited in the sense of minimization of the energy functional. In particular,  in \cite{AdaNoj15} it  was shown that $E(\mathbf{\Phi}^\alpha_{k}(\omega_k))<E(\bb \Phi^\alpha_{k+1}(\omega_{k+1}))$,  where $\omega_k$ and $\omega_{k+1}$ are such that  $||\bb \Phi^\alpha_{k}(\omega_{k})||_2=||\bb \Phi^\alpha_{k+1}(\omega_{k+1})||_2=m$, i.e. for a fixed mass constraint. Here $\mathbf{\Phi}^\alpha_{k}(\omega)$  stands for $\mathbf{\Phi}^\alpha_{k}$ (formally $\mathbf{\Phi}^\alpha_{k}$ is a function of $\omega$).

For $\alpha>0$  nothing is known about variational properties of  the profiles  $\mathbf{\Phi}^\alpha_k$.  In particular, one can easily verify that $S(\mathbf{\Phi}^\alpha_{0})>S(\mathbf{\Phi}^\alpha_{k})>S(\mathbf{\Phi}^\alpha_{k+1})$, \,$k\neq 0$.

\section{ The orbital stability of standing waves of the NLS-$\delta$   equation with attractive  nonlinearity}\label{att}
 Crucial role in
the stability analysis is played by the symmetries of  NLS equation \eqref{NLS_graph}. The basic symmetry associated to the  mentioned equation is  phase-invariance (in particular, translation invariance  does  not hold due to the defect at $\nu=0$). Thus,  it is reasonable  to define orbital stability as follows.
\begin{definition}\label{def_stabil}
 The standing wave $\mathbf U(t, x) = e^{i\omega t}\mathbf{\Phi}(x)$ is said to be \textit{orbitally stable} in a Hilbert space $\mathcal{H}$  if for any $\varepsilon > 0$ there exists $\eta > 0$ with the following property: if $\mathbf U_0 \in \mathcal{H}$  satisfies $||\mathbf U_0-\mathbf{\Phi}||_{\mathcal{H}} <\eta$,
then the solution $\mathbf U(t)$ of \eqref{NLS_graph} with $\mathbf U(0) = \mathbf U_0$ exists for any $t\in \mathbb{R}$, and
\[\sup\limits_{t\in \mathbb{R}}\inf\limits_{\theta\in\mathbb{R}}||\mathbf U(t)-e^{i\theta}\mathbf{\Phi}||_{\mathcal{H}} < \varepsilon.\]
Otherwise, the standing wave $\mathbf U(t, x) = e^{i\omega t}\mathbf{\Phi}(x)$ is said to be \textit{orbitally unstable} in $\mathcal{H}$.
\end{definition}
\subsection{Stability framework}\label{sub3.1}
  To formulate the stability theorem   for  NLS-$\delta$ equation  \eqref{NLS_graph} we will establish some basic objects. Let $\bb\Phi_k^\alpha$ be  defined by \eqref{Phi_k}. In what follows we will use the notation $\bb \Phi_k:=\bb\Phi_k^\alpha.$
We start  verifying that the profile $\bb \Phi_k$ is a critical point of the action functional $S$ defined by \eqref{S_graph}. Indeed, for $\bb U, \bb V\in \mathcal{E}$,
 \begin{equation*}
 \begin{split}
&S'(\bb U)\bb V=\frac{d}{dt}S(\bb U+t\bb V)|_{t=0}\\
&=\re\Big((\bb U',\bb V')_2+\omega(\bb U,\bb V)_2-(|\bb U|^{p-1}\bb U, \bb V)_2+\alpha u_1(0)\overline{v_1(0)}\Big).
 \end{split}
 \end{equation*}
Since  $\bb\Phi_k$ satisfies \eqref{H_alpha}, we get $S'(\bb \Phi_k)=0$.

In the approach by  \cite{GrilSha87, GrilSha90, oh} crucial role is played by spectral properties of the linear operator associated with the second derivative of $S$ calculated  at $\bb \Phi_k$ (linearization of \eqref{NLS_graph}). Thus, splitting   $\bb U, \bb V\in \mathcal{E}$ into real and imaginary parts $\mathbf{U}=\mathbf{U}^1+i\mathbf{U}^2$ and $\mathbf{V}=\mathbf{V}^1+i\mathbf{V}^2$, with the vector functions $\mathbf{U}^j,\mathbf{V}^j,\,j\in\{1,2\},$ being  real-valued, we get
\begin{equation*}\label{q_form}
\begin{split}
&S''(\bb \Phi_k)(\bb U, \bb V)\\
&=\Big[((\bb U^1)',(\bb V^1)')_2+\omega(\bb U^1,\bb V^1)_2-(p(\bb\Phi_k)^{p-1}\bb U^1, \bb V^1)_2+\alpha u^1_1(0)v^1_1(0)\Big]\\&+\Big[((\bb U^2)',(\bb V^2)')_2+\omega(\bb U^2,\bb V^2)_2-((\bb\Phi_k)^{p-1}\bb U^2, \bb V^2)_2+\alpha u^2_1(0)v^2_1(0)\Big].
\end{split}
\end{equation*}
Then it is easily seen that $S''(\bb \Phi_k)(\bb U, \bb V)$ can be formally rewritten as
\begin{equation}\label{BBS}
S''(\bb \Phi_k)(\bb U, \bb V)=B^\alpha_{1,k}(\bb U^1,\bb V^1)+B^\alpha_{2,k}(\bb U^2,\bb V^2).
\end{equation}
Here bilinear   forms  $B^\alpha_{1,k}$ and $B^\alpha_{2,k}$ are defined  for  $\bb F=(f_j)_{j=1} ^N, \bb G=(g_j)_{j=1} ^N\in \EE$  by
\begin{equation}\label{spec14}
\begin{split}
&B^\alpha_{1,k}(\bb F,\bb G)=\sum\limits_{j=1}^N\int\limits_0^\infty(f_j'g_j'+\omega f_jg_j-p(\varphi_{k,j})^{p-1}f_jg_j)dx+\alpha  f_1(0)g_1(0),\\
&B^\alpha_{2,k}(\bb F,\bb G)=\sum\limits_{j=1}^N\int\limits_0^\infty(f_j'g_j'+\omega f_jg_j-(\varphi_{k,j})^{p-1}f_jg_j)dx+\alpha  f_1(0)g_1(0),
\end{split}
\end{equation}
where $\varphi_{k,j}=\varphi_{k,j}^\alpha$.
 Next,  we determine  the self-adjoint operators associated with  the forms $B^\alpha_{j,k}$ in order to  establish a self-contained analysis.

First  note that the forms $B^\alpha_{j,k},\,j\in\{1,2\},$
are bilinear bounded from below and closed. Thus, there appear   self-adjoint operators  $\bb L^\alpha_{1,k}$  and $\bb L^\alpha_{2,k}$
associated (uniquely) with $B^\alpha_{1,k}$ and $B^\alpha_{2,k}$ by  the First Representation Theorem (see  \cite[Chapter VI, Section 2.1]{kato}), namely,
 \begin{equation}\label{bilinear}
 \begin{split}
 \begin{array}{ll}
  &\bb L^\alpha_{j,k}\bb V=\bb W,\qquad j\in\{1,2\},\\
 &\dom(\bb L^\alpha_{j,k})=\{\bb V\in \EE:\, \exists \bb W\in   L^2(\mathcal{G})\; s.t.\; \forall \bb Z\in \EE, \;B^\alpha_{j,k}(\bb V,\bb Z)=(\bb W, \bb Z)_2 \}.
  \end{array}
 \end{split}
\end{equation}
In the following theorem we describe the  operators $\bb L^\alpha_{1,k}$ and $\bb L^\alpha_{2,k}$ in more explicit
form.
\begin{theorem}\label{repres}
The operators  $\bb L^\alpha_{1,k}$ and $\bb L^\alpha_{2,k}$ defined by  \eqref{bilinear}  are given  on the domain $D_\alpha$ by
\begin{equation*}\label{NLSL9}
 \begin{split}
& \bb L^\alpha_{1,k}=\left(\Big(-\frac{d^2}{dx^2}+\omega-p(\varphi_{k,j})^{p-1}\Big)\delta_{i,j}\right), \\
& \bb L^\alpha_{2,k}=\left(\Big(-\frac{d^2}{dx^2}+\omega-(\varphi_{k,j})^{p-1}\Big)\delta_{i,j}\right),
 \end{split}
\end{equation*}
where $\delta_{i,j}$ is the Kronecker symbol.
\end{theorem}
\begin{proof}
Since the proof for $\bb L^\alpha_{2,k}$ is similar to the one for  $\bb L^\alpha_{1,k}$, we  deal with $\bb L^\alpha_{1,k}$. Let
$B^\alpha_{1,k}= B^\alpha + B_{1,k}$, where $B^{\alpha}: \EE \times \EE\to \mathbb R$ and $ B_{1,k}:L^2(\mathcal{G}) \times L^2(\mathcal{G})\to \mathbb R$ are  defined by
\begin{equation*}\label{repres2}
\begin{split}
&B^{\alpha}(\bb U,\bb V)=\sum\limits_{j=1}^N\int\limits_{0}^\infty u_j'v_j'dx+\alpha  u_1(0)v_1(0),\\
& B_{1,k}(\bb U,\bb V)=\sum\limits_{j=1}^N\int\limits_{0}^\infty(\omega-p(\varphi_{k,j})^{p-1})u_jv_jdx.
\end{split}
\end{equation*}
We denote by $\bb L^\alpha$ (resp. $\bb L_{1,k}$) the self-adjoint operator on
$L^2(\mathcal{G})$  associated (by the First Representation Theorem) with $B^{\alpha}$ (resp. $B_{1,k}$). Thus,
\begin{equation*}\label{Q1}
  \begin{array}{ll}
  \bb L^\alpha\bb V=\bb W,\\
\dom(\bb L^\alpha)=\{\bb V\in \EE: \exists \bb W\in  L^2(\mathcal{G})\; s.t.\; \forall \bb Z\in \EE, \;B^{\alpha}(\bb V,\bb Z)=(\bb W,\bb Z)_2 \}.
\end{array}
  \end{equation*}
The operator  $\bb L^\alpha$ is the self-adjoint extension of the following symmetric operator
\begin{equation*}\label{A0}
 \begin{split}
 &\bb L^{0}\bb V=(-v''_j(x))_{j=1}^N, \\
 &\dom(\bb L^{0})=\left\{\mathbf{V}\in H^2(\mathcal{G}):  v_1(0)=...=v_N(0)=0, \sum\limits_{j=1}^N  v_j'(0)=0 \right\}.
 \end{split}
\end{equation*}
Indeed, initially we have $\bb L^{0}\subset \bb L^\alpha$. Let $\bb V\in \dom(\bb L^{0})$ and   $\bb W= (-v''_j(x))_{j=1}^N\in L^2(\mathcal{G})$. Then for every $\bb Z\in \EE$ we get $B^{\alpha}(\bb V,\bb Z)=(\bb W,\bb Z)_2$.
Thus,  $\bb V\in \dom(\bb L^\alpha)$ and $\bb L^\alpha\bb V=\bb W=(-v''_j(x))_{j=1}^N$,  which yields the claim.

Arguing as in the  proof of Theorem \ref{spect_L^0_1}$(iii)$, we can show that  the deficiency indices of $\bb L^0$ are given by  $n_\pm(\bb L^0)=1$.  Therefore, there exists one-parametric family of self-adjoint extensions of $\bb L^{0}$. Similarly to   \cite[Theorem 3.1.1]{AlbGes05}, we can prove that all self-adjoint extensions of $\bb L^{0}$ are given by
\begin{equation*}\label{beta}
\begin{split}
 & \bb L^\beta\bb V=(-v''_j(x))_{j=1}^N,  \\
 &\dom(\bb L^\beta)=\left\{
 \bb V\in H^2(\mathcal{G}):  v_1(0)=...=v_N(0), \quad
 \sum\limits_{j=1}^N  v_j'(0)=\beta v_1(0),\,\,\beta\in \mathbb R\right\}.
 \end{split}
 \end{equation*}
To show this we assume that $\bb L^0$ acts on complex-valued functions. Then due to \cite[Theorem A.1]{AlbGes05}, any self-adjoint extension  $\widehat{\bb L}$ of $\bb L^0$ is defined by
$$\dom(\widehat{\bb L})=\left\{\bb F=\bb F_0+c\bb F_i+ce^{i\theta}\bb F_{-i}:\,\bb F_0\in \dom(\bb L^0), c\in\mathbb{C},\theta\in[0,2\pi)\right\},$$
where   $\bb F_{\pm i}=\left(\tfrac{i}{\sqrt{\pm i}}e^{i\sqrt{\pm i}x}\right)_{j=1}^N,\,\,\Im(\sqrt{\pm i})>0$.
It is easily seen that  for $\bb F\in \dom(\widehat{\bb L})$ we have
$$\sum\limits_{j=1}^N(\bb F)'_j(0)=-Nc(1+e^{i\theta}),\quad (\bb F)_j(0)=c\left(e^{i\pi/4}+e^{i(\theta-\pi/4)}\right).$$
From the last equalities it follows that
$$\sum\limits_{j=1}^N(\bb F)'_j(0)=\beta(\bb F)_1(0),\,\, \text{where}\,\, \beta=\frac{-N(1+e^{i\theta})}{\left(e^{i\pi/4}+e^{i(\theta-\pi/4)}\right)}\in\mathbb{R},$$ which induces that $\dom(\widehat{\bb L})\subseteq \dom(\bb L^\beta).$
Using the fact that $\bb L^\beta$ defined on $\dom(\bb L^\beta)$ is self-adjoint,   we arrive at $\dom(\widehat{\bb L})=\dom(\bb L^\beta)$  for some $\beta\in \mathbb{R}$.

Finally, we need to  prove that $\beta=\alpha$. Take $\bb V\in \dom(\bb L^\alpha)$ with $\bb V(0)\neq \bb 0$, then
we obtain $(\bb L^\alpha\bb V,\bb V)_2=\sum\limits_{j=1}^N\int\limits_{0}^\infty(v_j')^2dx+\beta  (v_1(0))^2$,
which should be equal to $B^{\alpha}(\bb V,\bb V)=\sum\limits_{j=1}^N\int\limits_{0}^\infty(v_j')^2dx+\alpha  (v_1(0))^2$ for all $\bb{V}\in \EE$.
Therefore, $\beta=\alpha$.

Note that $\bb L_{1,k}$ is the
 self-adjoint extension of the following  multiplication operator
\begin{equation*}
  \bb L_{0,k} \bb V =\Big((\omega-p(\varphi_{k,j})^{p-1})v_j(x)\Big)_{j=1}^N , \quad
 \dom(\bb L_{0,k})=L^2(\mathcal{G}).
 \end{equation*}
 Indeed, for $\bb V\in\dom(\bb L_{0,k})$  we define $\bb W= \Big((\omega-p(\varphi_{k,j})^{p-1})v_j(x)\Big)_{j=1}^N \in L^2(\mathcal{G})$. Then for every $\bb Z\in L^2(\mathcal{G})$ we get $B_{1,k}(\bb V,\bb Z)=(\bb W, \bb Z)_2$. Thus,
$\bb V\in \dom(\bb L_{1,k})$ and $\bb L_{1,k}\bb V=\bb W=\Big((\omega-p(\varphi_{k,j})^{p-1})v_j(x)\Big)_{j=1}^N$. Hence, $\bb L_{0,k}\subseteq \bb L_{1,k}$.  Since $\bb L_{0,k}$ is self-adjoint, $\bb L_{1,k}= \bb L_{0,k}$.
The Theorem is proved.
\end{proof}
\begin{remark}
The above theorem is the generalization of \cite[Lemma 10]{CozFuk08} (for $\mathcal{G}=\mathbb{R}$).
\end{remark}
It is easily seen from \eqref{BBS} that formally  $ S''(\mathbf{\Phi}^{\alpha}_k)$ can be considered as a self-adjoint $2N\times 2N$ matrix  operator  (see \cite{GrilSha87, GrilSha90} for the details)
\begin{equation*}\label{form_L1+L2}
\bb H^\alpha_k:=\left(\begin{array}{cc} \bb L^\alpha_{1,k}& \bb 0 \\ \bb 0 & \bb  L^\alpha_{2,k}  \end{array}\right).\end{equation*}

Define
 \begin{equation}\label{p(omega)}
 p(\omega_0)=\left\{\begin{array}{ll}
 1 &\, \text{if}\ \partial_\omega||\mathbf{\Phi}_k||^2_2>0\ \text{at}\ \omega=\omega_0, \\
 0 & \, \text{if}\ \partial_\omega||\mathbf{\Phi}_k||^2_2<0\ \text{at}\ \omega=\omega_0.
  \end{array}\right.
\end{equation}
Having established \emph{Assumptions 1, 2} in \cite{GrilSha87}, i.e. well-posedness of the associated Cauchy problem (see Theorem \ref{loc_well_posed}) and the existence of $C^1$ in $\omega$  standing wave,  the  next stability/instability result follows  from \cite[Theorem 3]{GrilSha87} and \cite[Corollary 3 and 4]{oh}.

\begin{theorem}\label{stabil_graph}
 Let $\alpha\neq 0$,\, $k\in \left\{1,...,\left[\tfrac{N-1}{2}\right]\right\}$,\, $\omega>\tfrac{\alpha^2}{(N-2k)^2}$,  and $n(\bb H^\alpha_k)$ be the number of negative eigenvalues of $\bb H^\alpha_k$. Suppose also  that

 $1)$\,\,$\ker(\bb L^\alpha_{2,k})=\Span\{\mathbf{\Phi}_k\}$;

 $2)$\,\,$\ker(\bb L^\alpha_{1,k})=\{\bb 0\}$;

  $3)$\,\, the negative spectrum of $\bb L^\alpha_{1,k}$ and  $\bb L^\alpha_{2,k}$ consists of a finite number of negative eigenvalues (counting multiplicities);

  $4)$\,\, the rest of  the  spectrum of $\bb L^\alpha_{1,k}$ and  $\bb L^\alpha_{2,k}$ is positive and bounded away from zero.
  \
  Then the following assertions hold.
\begin{itemize}
  \item[$(i)$] If $n(\bb H^\alpha_k)=p(\omega)=1$ in $L^2_k(\mathcal{G})$, then   the standing wave $e^{i\omega t}\mathbf{\Phi}_k$  is orbitally  stable in $\mathcal{E}_k$.
   \item[$(ii)$]  If $n(\bb H^\alpha_k)-p(\omega)=1$ in $L^2_k(\mathcal{G})$, then   the standing wave $e^{i\omega t}\mathbf{\Phi}_k$  is orbitally  unstable in $\mathcal{E}_k$, and therefore in $\EE$.
\end{itemize}
\end{theorem}

\begin{remark}\label{nonstab} The instability part of the above  Theorem is a  very delicate point  worth to be commented.
\begin{enumerate}
\item[$(i)$] It is  known from \cite{GrilSha90}  that when $n(\bb H^{\alpha}_k)-p(\omega)$ is odd, we obtain only spectral instability of $e^{i\omega t}\mathbf{\Phi}_k$. To obtain orbital instability due to \cite[Theorem 6.1]{GrilSha90}, it is sufficient to show estimate (6.2) in   \cite{GrilSha90} for the semigroup  $e^{t\bb A_{\alpha,k}}$ generated by
$$
\bb A_{\alpha,k}=\left(\begin{array}{cc} \bb 0& \bb L^\alpha_{2,k} \\ -\bb L^\alpha_{1,k} & \bb 0 \end{array}\right).
$$
In  our particular case it is not clear how to prove estimate (6.2).

\item[$(ii)$] In the case $n(\bb H^\alpha_k)=2$ (which happens in $L^2_k(\mathcal{G})$ for  $\alpha<0$)
we can apply the results by  Ohta \cite[Corollary 3 and 4]{oh} to get the instability part of the above Theorem.    We note that in this case the orbital instability follows without using spectral instability.

\item[$(iii)$] Generally,  to imply the orbital instability from the spectral one, the approach by    \cite{HenPer82} can be used (see Theorem 2).  The key point of this method is the use of the fact that  the mapping data-solution associated to the NLS-$\delta$ model is  of class $C^2$ as $p>2$ (see Theorem \ref{loc_well_posed}). For applications of the  approach by  \cite{HenPer82} to the models of KdV-type see \cite{AngLop08} and  \cite{AngNat16}.
\end{enumerate}
\end{remark}

\subsection{Spectral properties of $\bb L^\alpha_{1,k}$ and $\bb L^\alpha_{2,k}$}\label{sub3.2}
Below we  describe  the spectra of the operators $\mathbf{L}_{1,k}^\alpha$ and $\mathbf{L}_{2,k}^\alpha$  which will help us to verify the conditions of Theorem \ref{stabil_graph}. Our ideas are based on the extension theory of symmetric operators and the perturbation theory.

 The main result of this subsection is the following.

\begin{theorem}\label{7main}  Let $\alpha \neq 0$, $k\in\left\{1,...,\left[\tfrac{N-1}{2}\right]\right\}$ and  $\omega>\tfrac{\alpha^2}{(N-2k)^2}$. Then the following assertions hold.
\begin{enumerate}
\item[(i)]  If $\alpha<0$,  then $n(\bb H_k^\alpha)=2$  in  $L^2_k(\mathcal G)$, i.e. $n(\bb H_k^\alpha|_{L^2_k(\mathcal G)})=2$.

\item[(ii)]  If  $\alpha>0$, then $n(\bb H_k^\alpha)=1$  in  $L^2_k(\mathcal G)$, i.e. $n(\bb H_k^\alpha|_{L^2_k(\mathcal G)})=1$.
\end{enumerate}
\end{theorem}

Theorem \ref{7main} is an immediate consequence of  Propositions \ref{grafoN2} and \ref{n(L_1)} below.

 \begin{proposition}\label{grafoN2}
Let $\alpha\neq 0$,   $k\in\left\{1,...,\left[\tfrac{N-1}{2}\right]\right\}$ and  $\omega>\tfrac{\alpha^2}{(N-2k)^2}$. Then the following assertions hold.
\begin{itemize}
\item[$(i)$]   $\ker(\mathbf{L}^\alpha_{2,k})=\Span\{\mathbf{\Phi}_{k}\}$ and $\mathbf{L}^\alpha_{2,k}\geq 0$.
 \item[$(ii)$] $\ker(\mathbf{L}^\alpha_{1,k})=\{\mathbf{0}\}$.
  \item[$(iii)$] The positive part of the spectrum of the operators $\mathbf{L}^\alpha_{1,k}$ and $\mathbf{L}^\alpha_{2,k}$  is bounded away from zero, and $\sigma_{\ess}(\mathbf{L}^\alpha_{1,k})=\sigma_{\ess}(\mathbf{L}^\alpha_{2,k})=[\omega,\infty).$
  \end{itemize}
\end{proposition}

\begin{proof}
$(i)$\, It is obvious that  $\mathbf{\Phi}_k\in\ker(\mathbf{L}^\alpha_{2,k})$. To show the equality  $\ker(\mathbf{L}^\alpha_{2,k})=\Span\{\mathbf{\Phi}_k\}$  let us note that  any $\mathbf{V}=(v_j)_{j=1}^N\in H^2(\mathcal{G})$ satisfies the following identity
\begin{equation*}\label{identity_graph'}
-v''_j + \omega v_j-\varphi_{k,j}^{p-1}v_j=
\frac{-1}{\varphi_{k,j}}\frac{d}{dx}\left[\varphi_{k,j}^2\frac{d}{dx}\left(\frac{v_j}{\varphi_{k,j}}\right)\right],\quad x>0.
\end{equation*}
Thus, for $\mathbf{V}\in D_\alpha$ we obtain
\begin{equation*}\label{nonneg_graph'}
\begin{split}
&(\mathbf{L}^\alpha_{2,k}\mathbf{V},\mathbf{V})_2=
\sum\limits_{j=1}^N\int\limits^{\infty}_{0}(\varphi_{k,j})^2\left[\frac{d}{dx}\left(\frac{v_j}{\varphi_{k,j}}\right)\right]^2dx+
\sum\limits_{j=1}^N\left[-v_j'{v}_j+v_j^2\frac{(\varphi_{k,j})'}{\varphi_{k,j}}\right]^{\infty}_{0}\\&=\sum\limits_{j=1}^N\int\limits^{\infty}_{0}(\varphi_{k,j})^2\left[\frac{d}{dx}\left(\frac{v_j}{\varphi_{k,j}}\right)\right]^2dx+\sum\limits_{j=1}^N\left[v_j'(0){v_j(0)}-v_j^2(0)\frac{\varphi_{k,j}'(0)}{\varphi_{k,j}(0)}\right].
\end{split}
\end{equation*}
Using boundary conditions \eqref{D_alpha}, we get
\begin{equation*}
\begin{split}
&\sum\limits_{j=1}^N\left[v_j'(0){v_j(0)}-v_j^2(0)\frac{\varphi_{k,j}'(0)}{\varphi_{k,j}(0)}\right]\\
&=\alpha v_1^2(0)+\sqrt{\omega}v_1^2(0)\left[\sum\limits_{j=1}^k\tanh(-a_k)+\sum\limits_{j=k+1}^N\tanh(a_k)\right]\\&=\alpha v_1^2(0)+\sqrt{\omega}v_1^2(0)(N-2k)\frac{\alpha}{(2k-N)\sqrt{\omega}}=0,
\end{split}
\end{equation*}
which induces $(\mathbf{L}^\alpha_{2,k}\bb V, \bb V)_2> 0$ for $\bb V\in D_\alpha\setminus\Span\{\mathbf{\Phi}_k\}$. Therefore, $\ker(\mathbf{L}^\alpha_{2,k})=\Span\{\mathbf{\Phi}_k\}$.

$(ii)$\, Concerning the kernel of $\mathbf{L}^\alpha_{1,k}$,  the only $L^2(\mathbb{R}_+)$-solution of the equation
$$
 -v''_j+\omega v_j-p\varphi_{k,j}^{p-1}v_j=0
 $$
 is $v_j=\varphi_{k,j}'$ up to a factor. Thus, any element of $\ker(\mathbf{L}^\alpha_{1,k})$ has the form $\mathbf{V}=(v_j)_{j=1}^N=(c_j\varphi_{k,j}')_{j=1}^N,\, c_j\in\mathbb{R}$. Continuity condition $v_1(0)=...= v_N(0)$ induces that $c_1=...=c_N$, i.e.
 \begin{equation*}\label{kernel}
  v_j(x)=c\left\{\begin{array}{ll}
                     -\varphi_{k,j}', & \quad\hbox{$j=1,...,k$;} \\
                     \varphi_{k,j}', & \quad\hbox{$j=k+1,...,N$}
                    \end{array}
                  \right., \quad c\in\mathbb{R}.
                  \end{equation*}
                  Condition $\sum\limits_{j=1}^Nv_j'(0)=\alpha v_j(0)$ is equivalent to the equality $$c\left(\tfrac{\omega(1-p)}{2}+\tfrac{p-1}{2}\tfrac{\alpha^2}{(N-2k)^2}\right)=0.$$ The last one induces that either  $\omega=\tfrac{\alpha^2}{(N-2k)^2}$ (which is impossible) or $c=0$, and therefore  $\mathbf{V}\equiv\bb 0$.

$(iii)$ By Weyl's theorem (see \cite[Theorem XIII.14]{RS}) the essential spectrum of $\mathbf{L}^\alpha_{1,k}$ and $\mathbf{L}^\alpha_{2,k}$ coincides with $[\omega,\infty).$ Since  $\bb\Phi_k\in L^\infty(\mathcal{G})$ and $\bb \Phi_k(x)\to \bb 0$ as $x\to +\infty$, there can be only finitely many isolated eigenvalues  in $(-\infty, \omega')$ for any $\omega'<\omega$. Then $(iii)$ follows easily.
\end{proof}

 Below using the perturbation theory we will study
 $n(\mathbf{L}_{1,k}^\alpha)$ in the space $L^2_k(\mathcal{G})$ for any $k\in\{1,...,\left[\tfrac{N-1}{2}\right]\}$.
 For this purpose let us define the following  self-adjoint matrix Schr\"odinger operator on $L^2(\mathcal{G})$ with  the Kirchhoff condition at $\nu=0$
 \begin{equation}\label{L^0_1}
 \begin{split}
 &\bb L^0_1=\left(\Big(-\frac{d^2}{dx^2}+\omega-p\varphi_0^{p-1}\Big)\delta_{i,j}\right), \\
 &\dom(\bb L^0_1)=\left\{\mathbf{V}\in H^2(\mathcal{G}): v_1(0)=...=v_N(0),\,\,\sum\limits_{j=1}^N  v_j'(0)=0\right\},
 \end{split}
 \end{equation}
 where $\varphi_0$ represents the half-soliton solution for the classical NLS model,
 \begin{equation}\label{var_0}
 \varphi_0(x)=\left[\frac{(p+1)\omega}{2} \sech^2\left(\frac{(p-1)\sqrt{\omega}}{2}x \right)\right]^{\frac{1}{p-1}}.
  \end{equation}
From definition of the profiles $\bb \Phi^\alpha_{k}$ in \eqref{Phi_k} it follows
 \begin{equation*}\label{converg}
 \bb \Phi^\alpha_{k}\to\bb \Phi_0,\;\; \text{as}\;\; \alpha\to 0,\;\;\text{in}\;\; H^1(\mathcal G),
 \end{equation*}
 where $\bb \Phi_0=(\varphi_0)_{j=1}^N$.
 As we intend to study negative spectrum of $\mathbf{L}^\alpha_{1,k}$, we first need to describe spectral properties of  $\bb L^0_1$ (which is ``limit value" of $\mathbf{L}^\alpha_{1,k}$ as $\alpha\to 0$).

  \begin{theorem}\label{spect_L^0_1}
 Let $\bb L^0_1$ be defined by \eqref{L^0_1} and $k\in\left\{1,...,\left[\tfrac{N-1}{2}\right]\right\}$. Then the assertions below hold.
\begin{itemize}
  \item[$(i)$] $\ker(\bb L^0_1)=\Span\{\hat{\bb\Phi}_{0,1},...,\hat{\bb\Phi}_{0,N-1}\}$, where
   \begin{equation*}\label{Psi_j}
\hat{\mathbf{\Phi}}_{0,j}=(0,...,0,\underset{\bf j}{\varphi'_{0}},\underset{\bf j+1}{-\varphi'_{0}},0,...,0).
 \end{equation*}
  \item[$(ii)$] In the space  $L^2_k(\mathcal{G})$ we have $\ker(\bb L^0_1)=\Span\{\mathbf{\widetilde{\Phi}}_{0,k}\}$, i.e. $\ker(\bb L^0_1|_{L^2_k(\mathcal{G})})=\Span\{\mathbf{\widetilde{\Phi}}_{0,k}\}$, where
  \begin{equation}\label{Psi_0}
\mathbf{\widetilde{\Phi}}_{0,k}=\left(\underset{\bf 1}{\tfrac{N-k}{k}\varphi'_{0}},..., \underset{\bf k}{\tfrac{N-k}{k}\varphi'_{0}},\underset{\bf k+1}{-\varphi'_{0}},...,\underset{\bf N}{-\varphi'_{0}}\right).
 \end{equation}
  \item[$(iii)$]  The operator $\bb L^0_1$ has one simple negative eigenvalue in $L^2(\mathcal{G})$, i.e. $n(\bb L^0_1)=1$. Moreover, $\bb L^0_1$ has one simple negative eigenvalue in $L^2_k(\mathcal{G})$ for any $k$, i.e.  $n(\bb L^0_1|_{L^2_k(\mathcal{G})})=1$.
   \item[$(iv)$] The positive part of the spectrum of $\bb L^0_1$  is bounded away from zero, and $\sigma_{\ess}(\mathbf{L}^0_{1})=[\omega, \infty)$.
  \end{itemize}
\end{theorem}
\begin{proof}
The proof repeats the one of Theorem 3.12 in \cite{AngGol16}.

$(i)$ The only $L^2(\mathbb{R}_+)$-solution to the equation
$$
 -v''_j+\omega v_j-p\varphi_0^{p-1}v_j=0
 $$
 is $v_j=\varphi_0'$ (up to a factor). Thus, any element of $\ker(\mathbf{L}^0_1)$ has the form $\mathbf{V}=(v_j)_{j=1}^N=(c_j\varphi'_0)_{j=1}^N,\, c_j\in\mathbb{R}$. It is easily seen that continuity condition is satisfied since $\varphi'_0(0)=0$. Condition $\sum\limits_{j=1}^Nv_j'(0)=0$  gives rise to $(N-1)$-dimensional kernel of $\mathbf{L}^0_1$. It is obvious that functions $\hat{\mathbf{\Phi}}_{0,j},\, j=1,..., N-1$ form basis there.

  $(ii)$ Arguing as in the previous item, we can see that $\ker(\bb L^0_1)$ is one-dimensional in $L^2_k(\mathcal G)$, and it is spanned by $\mathbf{\widetilde{\Phi}}_{0,k}$.

  $(iii)$  In what follows  we will use the notation $\opl_0=\left(\Big(-\frac{d^2}{dx^2}+\omega-p\varphi_0^{p-1}\Big)\delta_{i,j}\right)$.
First,  note that $\mathbf{L}^0_1$ is the self-adjoint extension of the following symmetric operator
\begin{equation*}
\begin{split}
\mathbf{L}_0^0=\opl_0,\quad \dom(\mathbf{L}_0^0)=\left\{\mathbf{V}\in H^2(\mathcal{G}):  v_1(0)=...=v_N(0)=0, \sum\limits_{j=1}^N  v_j'(0)=0 \right\}.
\end{split}
\end{equation*}

Below we show that the operator $\mathbf{L}_0^0$ is non-negative and  has deficiency indices $n_\pm(\mathbf{L}_0^0)=1$. First, let us show that the   adjoint operator of $\mathbf{L}_0^0$ is given by
\begin{equation}\label{adjoint_graph}
\begin{split}
(\mathbf{L}_0^0)^*=\opl_0,\quad \dom((\mathbf{L}_0^0)^*)=\left\{\mathbf{V}\in H^2(\mathcal{G}): v_1(0)=...=v_N(0)\right\}.
\end{split}
\end{equation}
Using standard arguments one can prove that $\dom((\mathbf{L}_0^0)^*)\subset H^2(\mathcal{G})$ and  $(\mathbf{L}_0^0)^*=\opl_0$ (see  \cite[Chapter V,\S 17]{Nai67}).
Denoting
$$D_0^*:=\left\{\mathbf{V}\in H^2(\mathcal{G}): v_1(0)=...=v_N(0)\right\},$$ we easily get
the inclusion  $D_0^*\subseteq \dom((\mathbf{L}_0^0)^*)$. Indeed, let $\mathbf{U}=(u_j)_{j=1}^N\in D_0^*$ and $\bb V=(v_j)_{j=1}^N\in \dom(\mathbf{L}_0^0)$, then we get for $\mathbf{U}^*=\opl_0(\mathbf{U})\in L^2(\mathcal{G})$
\begin{equation*}
\begin{split}
&(\mathbf{L}_0^0\mathbf{V}, \mathbf{U})_2=(\opl_0(\mathbf{V}),\bb U)_2=(\bb V,\opl_0(\mathbf{U}))_2+\sum\limits_{j=1}^N\left[-v'_ju_j+v_ju'_j\right]_0^\infty\\&=(\bb V,\opl_0(\mathbf{U}))_2=(\bb V, \bb U^*)_2,
\end{split}
\end{equation*}
which, by definition of the adjoint operator, means that $\mathbf{U}\in \dom((\mathbf{L}_0^0)^*)$ or $D_0^*\subseteq \dom((\mathbf{L}_0^0)^*)$.

Let us show the inverse inclusion $D_0^*\supseteq \dom((\mathbf{L}_0^0)^*)$. Take $\bb U\in\dom((\mathbf{L}_0^0)^*)$, then for any $\bb V\in \dom(\mathbf{L}_0^0)$ we have
\begin{equation*}
\begin{split}
&(\mathbf{L}_0^0\bb V, \bb U)_2=(\opl_0(\bb V), \bb U)_2=(\bb V, \opl_0(\bb U))_2+\sum\limits_{j=1}^N\left[-v'_ju_j+v_ju'_j\right]_0^\infty\\&=(\bb V, (\mathbf{L}_0^0)^*\bb U)_2=(\bb V, \opl_0(\bb U))_2.
\end{split}
\end{equation*}
Thus, we arrive at the equality
\begin{equation}\label{adjoint}
\sum\limits_{j=1}^N\left[-v'_ju_j+v_ju'_j\right]_0^\infty=\sum\limits_{j=1}^N v'_j(0)u_j(0)=0
\end{equation}
for any $\bb V\in \dom(\mathbf{L}_0^0)$.
Let ${\bb W}=(w_j)_{j=1}^N\in \dom(\mathbf{L}_0^0)$ be such that $w'_3(0)= w'_4(0)=...= w'_N(0)= 0.$ Then   for  $\bb U\in \dom((\mathbf{L}_0^0)^*)$ from  \eqref{adjoint} it follows that
\begin{equation}\label{adjoint1} \sum\limits_{j=1}^N w_{j}'(0)u_j(0)=w'_1(0)u_1(0)+w'_2(0)u_2(0)=0.
\end{equation}
Recalling that $\sum\limits_{j=1}^N  w_j'(0)= w_1'(0)+w_2'(0)=0$ and assuming $w_2'(0)\neq 0$, we obtain from \eqref{adjoint1} the equality $u_1(0)=u_2(0)$.
 Repeating the similar arguments for ${\bb W}=({w}_j)_{j=1}^N\in \dom(\mathbf{L}_0^0)$ such that ${w}'_4(0)= {w}'_5(0)=...= {w}'_N(0)=0$, we  get $u_1(0)=u_2(0)=u_3(0)$ and so on. Finally taking ${\bb W}=({w}_j)_{j=1}^N\in \dom(\mathbf{L}_0^0)$ such that ${w}'_N(0)= 0$ we will arrive at $u_1(0)=u_2(0)=...=u_{N-1}(0)$ and consequently $u_1(0)=u_2(0)=...=u_{N}(0)$.
 Thus, $\bb U\in D_0^*$ or $D_0^*\supseteq \dom((\mathbf{L}_0^0)^*)$, and \eqref{adjoint_graph} holds.

Let us  show that the operator $\mathbf{L}_0^0$ is non-negative. First, note that every component of the vector $\mathbf{V}=(v_j)_{j=1}^N\in H^2(\mathcal{G})$ satisfies the following identity
\begin{equation*}\label{identity_graph}
-v_j''+\omega v_j-p\varphi_0^{p-1}v_j=
\frac{-1}{\varphi'_0}\frac{d}{dx}\left[(\varphi'_0)^2\frac{d}{dx}\left(\frac{v_j}{\varphi'_0}\right)\right],\quad x> 0.
\end{equation*}
Using the above equality and integrating by parts, we get for $\mathbf{V}\in \dom(\mathbf{L}_0^0)$
 \begin{equation*}\label{nonneg_graph}\begin{split}
&(\mathbf{L}_0^0\mathbf{V},\mathbf{V})_2=
\sum\limits_{j=1}^N\int\limits^{\infty}_{0}(\varphi'_{0})^2\left[\frac{d}{dx}\left(\frac{v_j}{\varphi'_{0}}\right)\right]^2dx+
\sum\limits_{j=1}^N\left[-v_j'{v}_j+v_j^2\frac{\varphi''_{0}}{\varphi'_{0}}\right]^{\infty}_{0}\\&=\sum\limits_{j=1}^N\int\limits^{\infty}_{0}(\varphi'_{0})^2\left[\frac{d}{dx}\left(\frac{v_j}{\varphi'_{0}}\right)\right]^2dx\geq 0,
\end{split}
\end{equation*}
where the non-integral  term becomes zero by the boundary conditions for $\mathbf{V}$ and the fact that $x=0$ is the first-order zero  for $\varphi'_0$ (i.e. $\varphi''_0(0)\neq 0$). Indeed,
$$
\sum\limits_{j=1}^N\left[-v_j'{v}_j+v_j^2\frac{\varphi''_{0}}{\varphi'_{0}}\right]^{\infty}_{0}=-\sum\limits_{j=1}^N\lim_{x\to 0+}\frac{2v_j(x)v'_j(x)\varphi''_0(x)+v_j^2(x)\varphi'''_0(x)}{\varphi''_{0}(x)}=0.
$$

Due to the von Neumann decomposition given in Proposition \ref{d5}, we obtain
\begin{equation*}
\dom((\mathbf{L}_0^0)^*)=\dom(\mathbf{L}_0^0)\oplus\Span\{\mathbf{V}_i\}\oplus\Span\{\mathbf{V}_{-i}\},
\end{equation*}
where $\mathbf{V}_{\pm i}=\left(e^{i\sqrt{\pm i}x}\right)_{j=1}^N,\, \Im(\sqrt{\pm i})>0$.
Indeed, since $\varphi_0\in L^\infty(\mathbb{R_+})$,  we get

$$\dom((\mathbf{L}_0^0)^*)=\dom(\mathbf{L}^*)=\dom(\mathbf{L})\oplus\Span\{\mathbf{V}_i\}\oplus\Span\{\mathbf{V}_{-i}\},$$ where
$$
\mathbf{L}=\left(\Big(-\frac{d^2}{dx^2}\Big)\delta_{i,j}\right), \quad \dom(\mathbf{L})=\dom(\mathbf{L}^0_0),\quad \mathcal{N}_{\pm}(\bb L)=\Span\{\mathbf{V}_{\pm i}\}.
$$
Since $n_\pm (\bb L)=1$, by \cite[Chapter IV, Theorem 6]{Nai67}, it follows that $n_{\pm}(\mathbf{L}_0^0)=1.$
Next, due to  Proposition \ref{semibounded},  $n(\mathbf{L}^0_{1})\leq 1$. Taking into account that $(\mathbf{L}^0_{1}\mathbf{\Phi}_0,\mathbf{\Phi}_0)_2=-(p-1)||\mathbf{\Phi}_0||_{p+1}^{p+1}<0$, where $\mathbf{\Phi}_0=(\varphi_{0})_{j=1}^N$, we  arrive at $n(\mathbf{L}^0_{1})=1$. Finally, since $\bb \Phi_0\in L_k^2(\mathcal{G})$ for any $k$, we have $n(\mathbf{L}^0_{1}|_{L_k^2(\mathcal{G})})=1$.

$(iv)$ follows from Weyl's theorem.
\end{proof}
\begin{remark}
Observe that, when we deal with the deficiency indices, the operator $\bb L_0^0$ is assumed to act on complex-valued functions which however does not affect the analysis of the negative spectrum of $\bb L_1^0$ acting on real-valued functions.
\end{remark}

The following lemma states the analyticity of the family  of operators $\mathbf{L}^\alpha_{1,k}$.

 \begin{lemma}\label{analici} As a function of $\alpha$, $(\mathbf{L}^\alpha_{1,k})$ is  real-analytic family of self-adjoint operators of type (B) in the sense of Kato.
\end{lemma}
\begin{proof} By Theorem \ref{repres} and \cite[Theorem VII-4.2]{kato}, it suffices to prove that the family  of bilinear forms $(B^{\alpha}_{1,k})$ defined in \eqref{spec14}  is   real-analytic  of type (B). Indeed,  it is immediate that  $B^{\alpha}_{1,k}$ is  bounded  from below and closed. Moreover, the decomposition  of $B^{\alpha}_{1,k}$ into $B^{\alpha}$ and $B_{1,k}$,  implies that $\alpha\to (B^{\alpha}_{1,k}\bb V, \bb V)$ is  analytic.
\end{proof}

Combining Lemma \ref{analici} and Theorem \ref{spect_L^0_1}, in the framework of the perturbation theory we obtain the following proposition.
\begin{proposition}\label{perteigen} Let  $k\in\left\{1,...,\left[\tfrac{N-1}{2}\right]\right\}$. Then there  exist $\alpha_0>0$ and two analytic functions $\lambda_k : (-\alpha_0,\alpha_0)\to \mathbb R$ and $\bb F_k: (-\alpha_0,\alpha_0)\to L^2_k(\mathcal{G})$ such that
\begin{enumerate}
\item[$(i)$] $\lambda_k(0)=0$ and $\bb F_k(0)=\mathbf{\widetilde{\Phi}}_{0,k}$, where $\mathbf{\widetilde{\Phi}}_{0,k}$ is defined by \eqref{Psi_0};

\item[$(ii)$] for all $\alpha\in (-\alpha_0,\alpha_0)$, $\lambda_k(\alpha)$ is the  simple isolated second eigenvalue of $\bb L^\alpha_{1,k}$ in $L^2_k(\mathcal{G})$, and $\bb F_k(\alpha)$ is the associated eigenvector for $\lambda_k(\alpha)$;

\item[$(iii)$] $\alpha_0$ can be chosen small enough to ensure that for  $\alpha\in (-\alpha_0,\alpha_0)$  the spectrum of $\bb L_{1,k}^\alpha$ in $L^2_k(\mathcal{G})$ is positive, except at most the  first two eigenvalues.
\end{enumerate}
\end{proposition}
\begin{proof} Using the  structure of the spectrum of the operator $\bb L^0_1$  given in   Theorem \ref{spect_L^0_1}$(ii)-(iv)$, we can separate the spectrum $\sigma(\bb L^0_1)$ in $L^2_k(\mathcal{G})$  into two parts $\sigma_0=\{\lambda^0_1, 0\}$, $\lambda^0_1<0$, and  $\sigma_1$ by a closed curve  $\Gamma$ (for example, a circle), such that $\sigma_0$ belongs to the inner domain of $\Gamma$ and $\sigma_1$ to the outer domain of $\Gamma$ (note that $\sigma_1\subset (\epsilon, +\infty)$ for $\epsilon>0$).  Next,  Lemma \ref{analici} and the analytic perturbations theory imply  that  $\Gamma\subset \rho(\bb L^\alpha_{1, k})$ for sufficiently small $|\alpha |$, and $\sigma (\bb L^\alpha_{1, k})$ is likewise separated by $\Gamma$ into two parts, such   that the part of $\sigma (\bb L^\alpha_{1, k})$ inside $\Gamma$ consists of a finite number of eigenvalues with total multiplicity (algebraic) two. Therefore, we obtain from the Kato-Rellich Theorem (see  \cite[Theorem XII.8]{RS}) the existence of two analytic functions $\lambda_k, \bb F_k$ defined in a neighborhood of zero such that   items $(i)$, $(ii)$ and $(iii)$ hold.
\end{proof}

Below we investigate how the perturbed second eigenvalue moves depending on the sign of $\alpha$.

\begin{proposition}\label{signeigen} There exists $0<\alpha_1<\alpha_0$ such that $\lambda_k(\alpha)<0$ for any $\alpha\in (-\alpha_1,0)$, and $\lambda_k(\alpha)>0$ for  any $\alpha\in (0, \alpha_1)$. Thus, in  $L^2_k(\mathcal G)$ for $\alpha$  small enough,  we have $n(\bb L^\alpha_{1, k})=2$  as $\alpha<0$, and $n(\bb L^\alpha_{1, k})=1$  as $\alpha>0$.
\end{proposition}

\begin{proof} From Taylor's theorem we have the following expansions
\begin{equation}\label{decomp1}
\lambda_k(\alpha)=\lambda_{0,k} \alpha+ O(\alpha^2)\quad\text{and}\quad \bb F_k(\alpha)=\mathbf{\widetilde{\Phi}}_{0,k}+ \alpha \bb F_{0,k}  +   \bb O(\alpha^2),
\end{equation}
where  $\lambda_{0,k}=\lambda'_k(0)\in \mathbb R$ and $\bb F_{0,k}=\partial_\alpha \bb F_k(\alpha)|_{\alpha=0}\in L^2_k(\mathcal{G})$.  The desired result will follow  if we show that $\lambda_{0,k}>0$.  We compute $(\bb L^\alpha_{1,k} \bb F_k(\alpha), \mathbf{\widetilde{\Phi}}_{0,k})_2$ in two different ways.

Note that for  $\bb \Phi_k=\bb \Phi^\alpha_k$ defined by \eqref{Phi_k} we have
\begin{equation}\label{1a}
\begin{split}
&\bb \Phi_k(\alpha)=\bb\Phi_0+\alpha\bb G_{0,k}+ \bb O(\alpha^2),\\ \bb G_{0,k}=\partial_\alpha\bb \Phi_k(\alpha)|_{\alpha=0}&=\tfrac{2}{(p-1)(N-2k)\omega}\left(\underset{\bf 1}{\varphi'_{0}},..., \underset{\bf k}{\varphi'_{0}},\underset{\bf k+1}{-\varphi'_{0}},...,\underset{\bf N}{-\varphi'_{0}}\right).
\end{split}
\end{equation}
 From \eqref{decomp1} we obtain
\begin{equation}\label{1}
(\bb L^\alpha_{1,k} \bb F_k(\alpha), \mathbf{\widetilde{\Phi}}_{0,k})_2=\lambda_{0,k} \alpha||\mathbf{\widetilde{\Phi}}_{0,k}||^2_2+ O(\alpha^2).
\end{equation}
By  $\bb L^0_{1} \mathbf{\widetilde{\Phi}}_{0,k}=\bb 0$ and \eqref{decomp1} we get
\begin{equation}\label{2}\bb L^\alpha_{1,k}\mathbf{\widetilde{\Phi}}_{0,k}=p\left((\bb \Phi_0)^{p-1}-(\bb \Phi_k)^{p-1}\right)\mathbf{\widetilde{\Phi}}_{0,k}=-\alpha p(p-1)(\bb \Phi_0)^{p-2}\bb G_{0,k}\mathbf{\widetilde{\Phi}}_{0,k}+ \bb O(\alpha^2).
\end{equation}
The operations in  the last equality are  componentwise.
Equations \eqref{2},  \eqref{1a}, and $\mathbf{\widetilde{\Phi}}_{0,k}\in D_\alpha$ induce
\begin{equation}\label{3}
\begin{split}
&(\bb L^\alpha_{1,k} \bb F_k(\alpha), \mathbf{\widetilde{\Phi}}_{0,k})_2
=(\bb F_k(\alpha), \bb L^\alpha_{1,k}\mathbf{\widetilde{\Phi}}_{0,k})_2\\&=-\left(\mathbf{\widetilde{\Phi}}_{0,k}, \alpha p(p-1)(\bb \Phi_0)^{p-2}\bb G_{0,k}\mathbf{\widetilde{\Phi}}_{0,k}\right)_2+O(\alpha^2)\\&=-\alpha p(p-1)\left(\tfrac{(N-k)^2}{k}-(N-k)\right)\tfrac{2}{(p-1)(N-2k)\omega}\int\limits_0^\infty(\varphi'_0)^3\varphi_0^{p-2}dx+O(\alpha^2)\\&=-2\alpha p\tfrac{N-k}{k\omega}\int\limits_0^\infty(\varphi'_0)^3\varphi_0^{p-2}dx+O(\alpha^2).
\end{split}
\end{equation}
Finally, combining \eqref{3} and \eqref{1}, we obtain
$$\lambda_{0,k}=\frac{-2p\tfrac{N-k}{k\omega}\int\limits_0^\infty(\varphi'_0)^3\varphi_0^{p-2}dx}{||\mathbf{\widetilde{\Phi}}_{0,k}||^2_2}+O(\alpha).$$
It follows that $\lambda_{0,k}$ is positive for sufficiently small $|\alpha|$ (due to the negativity of $\varphi'_0$ on $\mathbb{R}_+$), which in view of \eqref{decomp1} ends the proof.
\end{proof}

Now we can count the number of negative eigenvalues of $\bb L^\alpha_{1,k}$ in   $L^2_k(\mathcal{G})$ for any $\alpha$, using a classical continuation argument based on the Riesz projection.

\begin{proposition}\label{n(L_1)}
Let $k\in\left\{1,...,\left[\tfrac{N-1}{2}\right]\right\}$ and $\omega>\tfrac{\alpha^2}{(N-2k)^2}$. Then the following assertions hold.
\begin{enumerate}
\item[$(i)$] If  $\alpha>0$, then  $n(\bb L^\alpha_{1,k}|_{L^2_k(\mathcal{G})})=1$.
\item[$(ii)$] If $\alpha<0$, then  $n(\bb L^\alpha_{1,k}|_{L^2_k(\mathcal{G})})=2$.
\end{enumerate}
\end{proposition}

\begin{proof} We consider  the case $\alpha<0$. Recall that  $\ker(\bb L^\alpha_{1,k})=\{\bb 0\}$  by Proposition \ref{grafoN2}. Define $\alpha_\infty$ by
$$
\alpha_\infty=\inf \{\tilde \alpha<0:  \bb L^\alpha_{1,k}\;{\text{has exactly two negative eigenvalues for all}}\; \alpha \in (\tilde \alpha,0)\}.
$$
 Proposition \ref{signeigen} implies that $\alpha_\infty$ is well defined and $\alpha_\infty\in [-\infty,0)$. We claim that $\alpha_\infty=-\infty$. Suppose that $\alpha_\infty> -\infty$. Let $M=n(\bb L^{\alpha_{\infty}}_{1,k})$ and $\Gamma$  be a closed curve (for example, a circle or a rectangle) such that $0\in \Gamma\subset \rho(\bb L^{\alpha_{\infty}}_{1,k})$, and  all the negative eigenvalues of  $\bb L^{\alpha_{\infty}}_{1,k}$ belong to the inner domain of $\Gamma$.  The existence of such $\Gamma$ can be  deduced from the lower semi-boundedness of the quadratic form associated to $\bb L^{\alpha_{\infty}}_{1,k}$.

Next, from Lemma \ref{analici} it  follows  that there is  $\epsilon>0$ such that for $\alpha\in [\alpha_{\infty}-\epsilon, \alpha_{\infty}+\epsilon]$ we have $\Gamma\subset \rho(\bb L^\alpha_{1,k})$ and for $\xi \in \Gamma$,
$\alpha\to (\bb L^\alpha_{1,k}-\xi)^{-1}$ is analytic. Therefore, the existence of an analytic family of Riesz projections $\alpha\to P(\alpha)$  given by
$$
P(\alpha)=-\frac{1}{2\pi i}\int\limits_{\Gamma} (\bb L^{\alpha}_{1,k}-\xi)^{-1}d\xi
$$
implies  that $\dim(\ran P(\alpha))=\dim(\ran P(\alpha_\infty))=M$ for all $\alpha\in [\alpha_\infty-\epsilon, \alpha_{\infty}+\epsilon]$. Next, by definition of $\alpha_\infty$, $\bb L^{\alpha_{\infty}+\epsilon}_{1, k} $ has two negative eigenvalues and $M=2$, hence $\bb L^\alpha_{1,k}$ has two negative eigenvalues for $\alpha\in (\alpha_{\infty}-\epsilon, 0)$,  which contradicts with the definition of $\alpha_{\infty}$. Therefore,  $\alpha_{\infty}=-\infty$.
\end{proof}
\begin{remark}
\begin{itemize}
\item[$(i)$] The idea of using the continuation argument above was borrowed from \cite[Lemma 12]{CozFuk08}.
\item[$(ii)$]  We note that by Proposition \ref{grafoN'} in Appendix, the Morse index $n(\bb L^\alpha_{1,k})$ in the whole space $L^2(\mathcal{G})$ satisfies the  estimate $n(\bb L^\alpha_{1,k})\leq k+1$ for $\alpha<0$, and   $n(\bb L^\alpha_{1,k})\leq N-k$ for $\alpha>0$.
\end{itemize}
\end{remark}

\subsection{Slope analysis}\label{sub3.3}

In this subsection we evaluate $p(\omega)$ defined in \eqref{p(omega)}.
\begin{proposition}\label{slope_graph'}

Let $\alpha\neq 0$, $k\in\left\{1,...,\left[\tfrac{N-1}{2}\right]\right\}$, and $\omega>\tfrac{\alpha^2}{(N-2k)^2}$.  Let also  $J_k(\omega)=\partial_\omega||\mathbf{\Phi}_{k}^\alpha||_2^2$. Then the following assertions hold.
\begin{itemize}
\item[$(i)$] Let $\alpha<0$, then
\begin{itemize}
\item[$1)$] for $1<p\leq 5$, we have $J_k(\omega)>0$;
\item[$2)$] for $p> 5$, there exists $\omega_k^*$ such that $J_k(\omega_k^*)=0$, and  $J_k(\omega)>0$ for $\omega\in\left(\tfrac{\alpha^2}{(N-2k)^2},\omega_k^*\right)$, while  $J_k(\omega)<0$ for $\omega\in (\omega_k^*,\infty)$.
\end{itemize}
\item[$(ii)$] Let $\alpha>0$, then
\begin{itemize}
\item[$1)$] for $1<p\leq 3$, we have $J_k(\omega)>0$;
\item[$2)$] for $3<p<5$,  there exists $\hat{\omega}_k$ such that $J_k(\hat{\omega}_k)=0,$ and  $J_k(\omega)<0$ for $\omega\in\left(\tfrac{\alpha^2}{(N-2k)^2},\hat{\omega}_k\right)$, while  $J_k(\omega)>0$ for $\omega\in (\hat{\omega}_k,\infty)$;
\item[$3)$] for $p\geq 5$, we have $J_k(\omega)<0$.
\end{itemize}
\end{itemize}
\end{proposition}

\begin{proof}
Recall that $\mathbf{\Phi}_{k}^\alpha=(\varphi_{k,j}^\alpha)_{j=1}^N$, where $\varphi_{k,j}^\alpha$ is defined by \eqref{Phi_k}.
Changing variables, we get
\begin{equation*}
\begin{split}
&\int\limits_{0}^\infty(\varphi_{k,j}^\alpha(x))^2dx\\&=\left(\frac{p+1}{2}\right)^{\tfrac 2{p-1}}\frac{2\omega^{\tfrac 2{p-1}-\tfrac 1{2}}}{p-1}\left\{
                    \begin{array}{ll}
                    \int\limits_{\tanh^{-1}\left(\tfrac{-\alpha}{(2k-N)\sqrt{\omega}}\right)}^\infty \sech^{\tfrac{4}{p-1}}ydy,& \hbox{$j=1,...,k$};\\
    \int\limits_{\tanh^{-1}\left(\tfrac{\alpha}{(2k-N)\sqrt{\omega}}\right)}^\infty \sech^{\tfrac{4}{p-1}}ydy,& \hbox{$j=k+1,...,N$}
                    \end{array}
                    \right.\\
                    &=\left(\frac{p+1}{2}\right)^{\tfrac 2{p-1}}\frac{2\omega^{\tfrac 2{p-1}-\tfrac 1{2}}}{p-1}\left\{ \begin{array}{ll}
                    \int\limits_{\tfrac{-\alpha}{(2k-N)\sqrt{\omega}}}^1(1-t^2)^{\tfrac 2{p-1}-1}dt,& \hbox{$j=1,...,k$};\\
     \int\limits_{\tfrac{\alpha}{(2k-N)\sqrt{\omega}}}^1(1-t^2)^{\tfrac 2{p-1}-1}dt,& \hbox{$j=k+1,...,N.$}
                    \end{array}
                    \right.
\end{split}
\end{equation*}
Therefore, we obtain
\begin{equation*}
\begin{split}
&||\mathbf{\Phi}_{k}^\alpha||_2^2=\left(\frac{p+1}{2}\right)^{\tfrac 2{p-1}}\frac{2\omega^{\tfrac 2{p-1}-\tfrac 1{2}}}{p-1}[\int\limits_{\tfrac{-\alpha}{(2k-N)\sqrt{\omega}}}^1k(1-t^2)^{\tfrac 2{p-1}-1}dt\\&+\int\limits_{\tfrac{\alpha}{(2k-N)\sqrt{\omega}}}^1(N-k)(1-t^2)^{\tfrac 2{p-1}-1}dt].
\end{split}
\end{equation*}

From the last  equality we get
 \begin{equation}\label{J}
 \begin{split}
& J_k(\omega)=C\omega^{\tfrac {7-3p}{2(p-1)}}\tfrac{5-p}{p-1}[\int\limits_{\tfrac{-\alpha}{(2k-N)\sqrt{\omega}}}^1k(1-t^2)^{\tfrac {3-p}{p-1}}dt+\int\limits_{\tfrac{\alpha}{(2k-N)\sqrt{\omega}}}^1(N-k)(1-t^2)^{\tfrac {3-p}{p-1}}dt]\\&-C\omega^{\tfrac {7-3p}{2(p-1)}}\tfrac{\alpha}{\sqrt{\omega}}\left(1-\tfrac{\alpha^2}{(N-2k)^2\omega}\right)^{\tfrac{3-p}{p-1}}=C\omega^{\tfrac {7-3p}{2(p-1)}}\widetilde{J}_k(\omega),
 \end{split}
 \end{equation}
 where $C=\frac{1}{p-1}\left(\frac{p+1}{2}\right)^{\tfrac 2{p-1}}>0$ and
 \begin{equation*}
 \begin{split}
& \widetilde{J}_k(\omega)=\tfrac{5-p}{p-1}\left(\int\limits_{\tfrac{-\alpha}{(2k-N)\sqrt{\omega}}}^1k(1-t^2)^{\tfrac {3-p}{p-1}}dt+\int\limits_{\tfrac{\alpha}{(2k-N)\sqrt{\omega}}}^1(N-k)(1-t^2)^{\tfrac {3-p}{p-1}}dt\right)\\
 &-\tfrac{\alpha}{\sqrt{\omega}}\left(1-\tfrac{\alpha^2}{(N-2k)^2\omega}\right)^{\tfrac{3-p}{p-1}}.
 \end{split}
 \end{equation*}
Thus,
\begin{equation}\label{J'}
\widetilde{J}'_k(\omega)=\tfrac{-\alpha}{\omega^{3/2}}\tfrac{3-p}{p-1}\left[\left(1-\tfrac{\alpha^2}{(N-2k)^2\omega}\right)^{\tfrac{3-p}{p-1}}+\tfrac{\alpha^2}{(N-2k)^2\omega}\left(1-\tfrac{\alpha^2}{(N-2k)^2\omega}\right)^{-\tfrac{2(p-2)}{p-1}}\right].
\end{equation}
$(i)$ Let $\alpha<0$.
 It is immediate that $J_k(\omega)>0$ for $1< p\leq 5$ which yields $1)$. Consider the  case $p>5$.  It is easily seen that
$$\lim\limits_{\omega\to \tfrac{\alpha^2}{(N-2k)^2}}\widetilde{J}_k(\omega)=\infty,\quad \lim\limits_{\omega\to \infty}\widetilde{J}_k(\omega)=\tfrac{5-p}{p-1}N \int\limits_{0}^1(1-t^2)^{\tfrac {3-p}{p-1}}dt<0.$$
 Moreover, from \eqref{J'} it follows that $\widetilde{J}'_k(\omega)<0$ for $\omega> \tfrac{\alpha^2}{(N-2k)^2}$ and consequently $J_k(\omega)$ is strictly decreasing. Therefore, there exists a unique $\omega_k^*>\tfrac{\alpha^2}{(N-2k)^2}$ such that $\widetilde{J}_k(\omega_k^*)=J_k(\omega_k^*)=0$, consequently  $J_k(\omega)>0$ for $\omega\in\left(\tfrac{\alpha^2}{(N-2k)^2},\omega_k^*\right)$
 and $J_k(\omega)<0$ for $\omega\in(\omega_k^*, \infty)$, and the proof of $(i)-2)$ is completed.

 $(ii)$ Let $\alpha>0$. It is easily seen that $\widetilde{J}_k(\omega)<0$  for $p\geq 5$, thus, $3)$ holds.
Let $1<p<5$. It can be easily verified that
\begin{equation}\label{slo4}
 \lim_{\omega\to \infty}\widetilde{J}_k(\omega)=\frac{5-p}{p-1}N\int_0^{1}(1-t^2)^{\tfrac {3-p}{p-1}}dt>0,
 \end{equation}
and
\begin{equation}\label{slo5}
\begin{split}
\lim_{\omega\to \frac{\alpha^2}{(N-2k)^2}}\widetilde{J}_k(\omega)=\left\{
                     \begin{array}{ll}
                      \frac{5-p}{p-1}(N-k)\int_{-1}^{1}(1-t^2)^{\tfrac {3-p}{p-1}}dt>0, & \hbox{$p\in (1, 3]$,} \\
                       -\infty, & \hbox{$p\in (3,5)$.}
                     \end{array} \right.
\end{split}
\end{equation}
 Let $1<p\leq 3$, using the fact that  $\widetilde{J}'_k(\omega)<0$ we get from \eqref{slo4}-\eqref{slo5} the inequality $J_k(\omega)>0$, and $(ii)-1)$ holds.
 Let $3<p<5$, then $\widetilde{J}'_k(\omega)>0$, therefore, from \eqref{slo4}-\eqref{slo5} it follows that there exists $\hat{\omega}_k>\tfrac{\alpha^2}{(N-2k)^2}$ such that $\widetilde{J}_k(\hat{\omega}_k)=J_k(\hat{\omega}_k)=0$, moreover,  $J_k(\omega)<0$ for $(\tfrac{\alpha^2}{(N-2k)^2}, \hat{\omega}_k)$, and $J_k(\omega)>0$ for $(\hat{\omega}_k, \infty)$, i.e. $(ii)-2)$ is proved.
   \end{proof}

   \vspace*{4pt}\noindent{\it Proof of Theorem \ref{main}.}
$(i)$
 Let $\alpha<0$.
    Due to Theorem \ref{7main}, we have $n(\bb H^\alpha_{k})=2$ in $L_k^2(\mathcal{G})$. Therefore, by Proposition \ref{slope_graph'}$(i)$ we obtain
$$
n(\bb H_{k}^\alpha)-p(\omega)=1
$$
for $1<p\leq 5,\omega>\tfrac{\alpha^2}{(N-2k)^2}$,  and for $p>5$, $\omega\in (\tfrac{\alpha^2}{(N-2k)^2},\omega^*_k).$
Thus, from Theorem \ref{stabil_graph} (see also  Remark \ref{nonstab}$(ii)-(iii)$) we get the assertions $(i)-1)$ and $(i)-2)$ in $\EE_k$. Since $\EE_k\subset\EE$, we get the results in $\EE$.

$(ii)$ Let $\alpha>0$. Due to Theorem \ref{7main}, we have $n(\bb H^\alpha_{k})=1$ in $L_k^2(\mathcal{G})$. Therefore, by Proposition \ref{slope_graph'}$(ii)$ we obtain
$$
n(\bb H_{k}^\alpha)-p(\omega)=1
$$
for $p\geq 5, \omega>\tfrac{\alpha^2}{(N-2k)^2}$ and $3<p<5, \omega\in (\tfrac{\alpha^2}{(N-2k)^2},\hat{\omega}_k)$. Hence  we get instability of $e^{i\omega t}\mathbf{\Phi}_{k}^\alpha$ in $\EE_k$ and consequently in $\EE$.
From the other hand, for $1<p\leq 3,\omega>\tfrac{\alpha^2}{(N-2k)^2}$  and $3<p<5, \omega\in(\hat{\omega}_k,\infty)$, we have
$$
n(\bb H_{k}^\alpha)=p(\omega)=1,
$$
which yields stability of $e^{i\omega t}\mathbf{\Phi}_{k}^\alpha$ in $\EE_k$. Thus, $(ii)$ is proved.
 \hfill$\square$
 \begin{remark} \label{p}
 \begin{enumerate}
\item[$(i)$] Let $p>5$, $\alpha<0$ and $\omega>\omega^*_k$, then  $
n(\bb H_{k}^\alpha)-p(\omega)=2$ in $L^2_k(\mathcal{G})$, and therefore Theorem \ref{stabil_graph} does not provide any information about stability of $\bb \Phi_k$.

\item[$(ii)$] Let  $p>3$, then the orbital instability  results follow from the spectral instability of $\bb \Phi_k$ applying the approach by \cite{HenPer82} (see Remark \ref{nonstab}$(iii)$).
 \end{enumerate}
\end{remark}

\subsection{The  Kirchhoff  condition}\label{Kirch}
\

\noindent{\it Proof of Theorem \ref{Kirsh}.} The action functional for $\alpha=0$ has the form
$$ S_0(\mathbf{V})=\tfrac 1{2}||\mathbf{V}'||_2^2+\tfrac{\omega}{2}||\mathbf{V}||_2^2 -\tfrac 1{p+1}||\mathbf{V}||_{p+1}^{p+1}, \quad \mathbf{V}=(v_j)_{j=1}^N\in\EE.$$
Then $S''_0(\bb \Phi_0)=:\bb H_0=\left(\begin{array}{cc}
\bb L_1^0& \bb 0\\
\bb 0& \bb L_2^0
\end{array}\right),$ where
\begin{equation*}\label{L_1-L_2}
 \begin{split}
 &\bb L^0_1=\left(\Big(-\frac{d^2}{dx^2}+\omega-p\varphi_0^{p-1}\Big)\delta_{i,j}\right), \quad \bb L^0_2=\left(\Big(-\frac{d^2}{dx^2}+\omega-\varphi_0^{p-1}\Big)\delta_{i,j}\right),\\
 &\dom(\bb L^0_1)=\dom(\bb L^0_2)=\left\{\mathbf{V}\in H^2(\mathcal{G}): v_1(0)=...=v_N(0),\,\,\sum\limits_{j=1}^N  v_j'(0)=0\right\}.
 \end{split}
 \end{equation*}
 Our idea is to apply the stability Theorem \ref{stabil_graph} (substituting $\bb L^\alpha_{1,k}$ and  $\bb L^\alpha_{2,k}$ by  $\bb L^0_{1}$ and  $\bb L^0_{2}$ respectively, and $\bb \Phi_k$ by $\bb \Phi_0$).

 The spectrum of  $\bb L^0_{1}$ has been studied in Theorem \ref{spect_L^0_1}.  Note that in $L_{\eq}^2(\mathcal{G})$ the kernel of  $\bb L^0_{1}$  is empty, moreover, $n(\bb L^0_{1}|_{L_{\eq}^2(\mathcal{G})})=1$ since  $\bb \Phi_0\in L_{\eq}^2(\mathcal{G})$ and $(\bb L^0_{1}\bb \Phi_0, \bb \Phi_0)_2<0$.
 It is easy to show that   $\bb L^0_{2}\geq 0$ and $\ker(\bb L^0_{2})=\Span\{\bb \Phi_0\}$ (see the proof of Proposition \ref{grafoN2}$(i)$).

 To complete the proof we need to study the sign of $\partial_\omega||\bb \Phi_0||_2^2$. From \eqref{J} for $k=0$ and $\alpha=0$ it follows that
 $$\partial_\omega||\bb \Phi_0||_2^2=\tfrac{N}{p-1}\left(\tfrac{p+1}{2}\right)^{\tfrac 2{p-1}}\omega^{\tfrac {7-3p}{2(p-1)}}\tfrac{5-p}{p-1}\int\limits_{0}^1(1-t^2)^{\tfrac {3-p}{p-1}}dt,$$
 which is obviously positive for $1<p<5$, and is negative for $p>5$.
 Finally, using  $n(\bb H_0|_{L_{\eq}^2(\mathcal{G})})=1$, by Theorem \ref{stabil_graph}, for $1<p<5$ we get stability of $e^{i\omega t}\bb \Phi_0(x)$ in $\EE_{\eq}$, and  for $p>5$  instability of
$e^{i\omega t}\bb \Phi_0(x)$  in $\EE_{\eq}$ and consequently in $\EE$.
\hfill$\square$
\begin{remark}
\begin{itemize}
\item[$(i)$] An interesting connection with a problem on the line is due to the fact that the space $L^2_{\eq}(\mathcal{G})$  is similar to the one studied in \cite{DeiPar11}.
\item[$(ii)$] Note that the orbital instability part of the above theorem follows from the spectral instability since $p>5$ (see Remark \ref{nonstab}$(iii)$).
\end{itemize}
\end{remark}

 \section{ The orbital stability of standing waves of the NLS-$\delta$   equation with repulsive nonlinearity}\label{repuls}\label{rep}

In this section we study the  orbital stability of the standing waves of the  NLS-$\delta$   equation with repulsive nonlinearity ($\mu=-1$ in \eqref{NLS_graph}). The case $\mathcal{G}=\mathbb{R}$ was considered in \cite{KamOht09}. The profile $\bb \Phi(x)$ of the standing wave $e^{i\omega t}\bb\Phi(x)$ satisfies the equation
\begin{equation}\label{H_alpha_rep}
\mathbf{H}^\alpha_\delta\mathbf{\Phi}+\omega\mathbf{\Phi}+|\mathbf{\Phi}|^{p-1}\mathbf{\Phi}=0,\quad \bb \Phi\in D_\alpha.
\end{equation}
Equivalently $\bb\Phi$ is a critical point of the  action functional defined as
\begin{equation*}\label{S_rep}
 S_{\rep}(\mathbf{V})=\tfrac 1{2}||\mathbf{V}'||_2^2+\tfrac{\omega}{2}||\mathbf{V}||_2^2 +\tfrac 1{p+1}||\mathbf{V}||_{p+1}^{p+1}+\tfrac{\alpha}{2}|v_1(0)|^2, \quad \mathbf{V}=(v_j)_{j=1}^N\in\EE.
 \end{equation*}
In the following theorem we describe the solutions to equation \eqref{H_alpha_rep}.
 \begin{theorem}
 Let $\alpha<0$  and $0<\omega< \tfrac{\alpha^2}{N^2}$. Then equation  \eqref{H_alpha_rep} has the unique solution (up to permutations of the  edges of $\mathcal{G}$)
 $\bb \Phi_\alpha=(\varphi_{\alpha})_{j=1}^N$, where
 \begin{equation*}\label{varphi_rep}
 \varphi_\alpha(x)=\left[\frac{(p+1)\omega}{2} \csch^2\left(\frac{(p-1)\sqrt{\omega}}{2}x+\coth^{-1}\left(\frac{-\alpha}{N\sqrt{\omega}}\right)\right)\right]^{\frac{1}{p-1}}.
 \end{equation*}
 \end{theorem}
\begin{proof}
Notice that $\mathbf{H}^\alpha_\delta$ acts componentwise  as the Laplacian, thus  if $\bb \Phi=(\varphi_j)_{j=1}^N$ is the solution to \eqref{H_alpha_rep}, then $\varphi_j$ is the $L^2(\mathbb{R}_+)$-solution to the equation
\begin{equation}\label{psi_rep}
-\varphi_j''+\omega \varphi_j + |\varphi_j|^{p-1} \varphi_j = 0.
\end{equation}
The most general $L^2 (\mathbb{R}_+)$-solution  to \eqref{psi_rep} is $$\varphi(x) =\sigma\left[\frac{(p+1)\omega}{2} \csch^2\left(\frac{(p-1)\sqrt{\omega}}{2}x+y\right)\right]^{\frac{1}{p-1}},$$
where $\sigma \in \mathbb{C}$, $|\sigma|=1$ and $y\in\mathbb{R}$ (see \cite{KamOht09}). Therefore, the components $\varphi_j$
of the solution   $\bb\Phi$ to \eqref{H_alpha_rep} are given by
\begin{equation*}
\varphi_j(x)=\sigma_j\left[\frac{(p+1)\omega}{2} \csch^2\left(\frac{(p-1)\sqrt{\omega}}{2}x+y_j\right)\right]^{\frac{1}{p-1}}.
\end{equation*}
In order to solve \eqref{H_alpha_rep} we need to impose boundary conditions  \eqref{D_alpha}.  The
continuity condition in \eqref{D_alpha} and existence of the limits $\lim\limits_{x\rightarrow 0+}\varphi_j(x)$ imply that  $y_1=...=y_N= a>0$ and $\sigma_1=\ldots=\sigma_N=\sigma$. We can omit the dependence on $\sigma$ without losing generality.
 The second boundary condition in
\eqref{D_alpha} rewrites as
\begin{equation}\label{eq-a}
N\coth(a)=\frac{-\alpha}{\sqrt{\omega}}.
\end{equation}
From equation  \eqref{eq-a} it follows that $0<\omega< \tfrac{\alpha^2}{N^2}$ and $a=\coth^{-1}\left(\frac{-\alpha}{N\sqrt{\omega}}\right)$.
\end{proof}

 \begin{remark} Note that, in contrast to the NLS-$\delta$ equation with focusing nonlinearity, the solution to \eqref{H_alpha_rep} does not exist for $\alpha\geq 0$ due to the fact that the parameter $a$ in \eqref{eq-a} has to be positive to guarantee the existence of $\lim\limits_{x\rightarrow 0+}\varphi_j(x)$.
 \end{remark}

  \vspace*{4pt}\noindent{\it Proof of Theorem \ref{main_rep}.}
The proof of the particular case $\mathcal{G}=\mathbb{R}$ was given in \cite{KamOht09}.

The global well-posedness  of the Cauchy problem  for  $\mu=-1$  follows from Lemma \ref{glob_well}.
Analogously to the previous case, the second variation of $S_{\rep}$ at $\mathbf{\Phi_\alpha}$ can be written formally
\begin{equation}\label{BBSrep}
(S_{\rep})''(\mathbf{\Phi}_\alpha)=:\bb H_{\rep}^\alpha = \left(\begin{array}{cc} \mathbf{L}_{1,\rep}^\alpha & \bb 0 \\ \bb 0 &\mathbf{L}_{2,\rep}^\alpha \end{array} \right),
\end{equation}
with
 \begin{equation*}\label{linear_alpha}
\begin{split} &\mathbf{L}_{1,\rep}^\alpha=\left(\Big(-\frac{d^2}{dx^2}+\omega+p\varphi_{\alpha}^{p-1}\Big)\delta_{i,j}\right), \\ &\mathbf{L}_{2,\rep}^\alpha=\left(\Big(-\frac{d^2}{dx^2}+\omega+\varphi_{\alpha}^{p-1}\Big)\delta_{i,j}\right),\,\, \dom(\mathbf{L}_{1,\rep}^\alpha)=\dom(\mathbf{L}_{2,\rep}^\alpha)=D_\alpha,
\end{split}
\end{equation*}
where $\delta_{i,j}$ is the Kronecker symbol.

Let us show that $\ker(\mathbf{L}_{2,\rep}^\alpha)=\Span\{\mathbf{\Phi}_\alpha\}$. It is obvious that $\mathbf{\Phi}_\alpha\in \ker(\mathbf{L}_{2,\rep}^\alpha)$. Any $\mathbf{V}=(v_j)_{j=1}^N\in D_\alpha$ satisfies the following identity
\begin{equation*}\label{identity_graph'}
-v_j''+\omega v_j+\varphi_{\alpha}^{p-1}v_j=
\frac{-1}{\varphi_{\alpha}}\frac{d}{dx}\left[\varphi_{\alpha}^2\frac{d}{dx}\left(\frac{v_j}{\varphi_{\alpha}}\right)\right],\quad x>0.
\end{equation*}
Then we get for any $\mathbf{V}=(v_j)_{j=1}^N\in D_\alpha\setminus\Span\{\mathbf{\Phi}_\alpha\}$
\begin{equation*}
\begin{split}
&(\mathbf{L}_{2,\rep}^\alpha\mathbf{V},\mathbf{V})_2=
\sum\limits_{j=1}^N\int\limits^{\infty}_{0}\varphi_{\alpha}^2\left[\frac{d}{dx}\left(\frac{v_j}{\varphi_{\alpha}}\right)\right]^2dx+
\sum\limits_{j=1}^N\left[-v_j'v_j+v_j^2\frac{\varphi_{\alpha}'}{\varphi_{\alpha}}\right]^{\infty}_{0}\\
&=\sum\limits_{j=1}^N\int\limits^{\infty}_{0}\varphi_{\alpha}^2\left[\frac{d}{dx}\left(\frac{v_j}{\varphi_{\alpha}}\right)\right]^2dx+\sum\limits_{j=1}^N\left[v_j'(0)v_j(0)-v_j^2(0)\frac{\varphi_{\alpha}'(0)}{\varphi_{\alpha}(0)}\right]\\&=\sum\limits_{j=1}^N\int\limits^{\infty}_{0}\varphi_{j}^2\left[\frac{d}{dx}\left(\frac{v_j}{\varphi_{\alpha}}\right)\right]^2dx> 0.
\end{split}
\end{equation*}
Thus, $\ker(\mathbf{L}_{2,\rep}^\alpha)=\Span\{\mathbf{\Phi}_\alpha\}$.
The  inequality
$$
(\mathbf{L}_{1,\rep}^\alpha\bb V,\bb V)_2>(\mathbf{L}_{2,\rep}^\alpha\bb V,\bb V)_2,\quad \bb V\in D_\alpha\setminus\{\bb 0\},
$$
implies immediately that  $\mathbf{L}_{1,\rep}^\alpha\geq 0$ and $\ker(\mathbf{L}_{1,\rep}^\alpha)=\{\bb 0\}$.

By Weyl's theorem,  the essential spectrum of $\mathbf{L}_{1,\rep}^\alpha$ and $\mathbf{L}_{2,\rep}^\alpha$ coincides with $[\omega,\infty)$. Moreover,  there can be only finitely many isolated eigenvalues  in $(-\infty, \omega)$. Thus, except the zero eigenvalue of $\mathbf{L}_{2,\rep}^\alpha$, the spectrum of $\mathbf{L}_{1,\rep}^\alpha$ and $\mathbf{L}_{2,\rep}^\alpha$ is positive and  bounded away from zero. Therefore,  using the  classical Lyapunov analysis and noting that  $\bb H_{\rep}^\alpha$ is non-negative due to \eqref{BBSrep}, we obtain  that $e^{i\omega t}\bb \Phi_\alpha$ is orbitally stable.
\hfill$\square$
 \section*{Appendix}\label{app}
For convenience of the reader  we formulate the following two  results from  the extension theory (see \cite{Nai67}) essentially used in  our stability analysis. The first one reads as follows.

\begin{proposition}\label{d5} (von Neumann decomposition)
Let $A$ be a closed densely defined  symmetric operator. Then the following decomposition holds
\begin{equation}\label{d6}
\dom(A^*)=\dom(A)\oplus\mathcal N_{+}(A)\oplus\mathcal N_{-}(A).
\end{equation}
Therefore, for $u\in \dom(A^*)$ such that  $u=f+f_i+f_{-i}$, with $f\in \dom(A), f_{\pm i}\in \mathcal{N}_{\pm}(A)$, we get
\begin{equation*}\label{d6a}
A^*u=Af+if_i-if_{-i}.
\end{equation*}
\end{proposition}
\begin{remark} The direct sum in \eqref{d6} is not necessarily orthogonal.
\end{remark}
\begin{proposition}\label{semibounded}
Let $A$  be a densely defined lower semi-bounded symmetric operator (that is, $A\geq mI$)  with finite deficiency indices $n_{\pm}(A)=n<\infty$  in the Hilbert space $\mathcal{H}$, and let $\widetilde{A}$ be a self-adjoint extension of $A$.  Then the spectrum of $\widetilde{A}$  in $(-\infty, m)$ is discrete and  consists of at most $n$  eigenvalues counting multiplicities.
\end{proposition}

Below, using the above abstract results,  we provide  an estimate for  the Morse index of the operator  $\mathbf{L}^\alpha_{1,k}$ defined in Theorem \ref{repres} in the whole space $L^2(\mathcal{G})$.

 \begin{proposition}\label{grafoN'}
Let $\alpha\neq 0$,   $k\in\left\{1,...,\left[\tfrac{N-1}{2}\right]\right\}$ and  $\omega>\tfrac{\alpha^2}{(N-2k)^2}$. Then the following assertions hold.
\begin{itemize}
\item[$(i)$] If  $\alpha< 0$, then $n(\mathbf{L}^\alpha_{1,k})\leq k+1$.
\item[$(ii)$] If $\alpha>0$, then $n(\mathbf{L}^\alpha_{1,k})\leq N-k$.
  \end{itemize}
\end{proposition}
\begin{proof}  $(i)$ In what follows  we will use the notation $$\opl^\alpha_k=\left(\Big(-\frac{d^2}{dx^2}+\omega-p(\varphi_{k,j})^{p-1}\Big)\delta_{i,j}\right).$$
First,  note that $\mathbf{L}^\alpha_{1,k}$ is the self-adjoint extension of the following symmetric operator
\begin{equation*}
\begin{split}
\mathbf{\widetilde L}_{0,k}=\opl^\alpha_k,\quad \dom(\mathbf{\widetilde L}_{0,k})=\left\{\begin{array}{c}
\bb V\in H^2(\mathcal{G}):  v_1(0)=...=v_N(0)=0,\\
\sum\limits_{j=1}^N  v_j'(0)=0,\, v_1(b_k)=...=v_k(b_k)=0
\end{array}\right\},
\end{split}
\end{equation*}
where $b_k=\frac{2}{(p-1)\sqrt{\omega}}a_k$ and $a_k$ is defined in Theorem \ref{1bump}.
Below we show that the operator $\mathbf{\widetilde L}_{0,k}$ is non-negative, and  $n_\pm(\mathbf{\widetilde L}_{0,k})=k+1$. Let us  show that the   adjoint operator of $\mathbf{\widetilde L}_{0,k}$ is given by
\begin{equation}\label{adjoint_graph'}
\begin{split}
&\mathbf{\widetilde L}_{0,k}^*=\opl^\alpha_k,\\ &\dom(\mathbf{\widetilde L}_{0,k}^*)=\left\{\begin{array}{c}\bb{V}\in L^2(\mathcal{G}): v_1(0)=...=v_N(0),\, v_{k+1},...,v_N\in H^2(\mathbb{R}_+),\\v_1,...,v_k\in H^2(\mathbb{R}_+\setminus\{b_k\})\cap H^1(\mathbb{R}_+)
\end{array}\right\}.
\end{split}
\end{equation}
Using standard arguments, one can prove that   $\mathbf{\widetilde L}_{0,k}^*=\opl^\alpha_k$ (see  \cite[Chapter V, \S 17]{Nai67}).
We denote $$D_{0,k}^*:=\left\{\begin{array}{c}\bb{V}\in L^2(\mathcal{G}): v_1(0)=...=v_N(0),\, v_{k+1},...,v_N\in H^2(\mathbb{R}_+),\\v_1,..,v_k\in H^2(\mathbb{R}_+\setminus\{b_k\})\cap H^1(\mathbb{R}_+)
\end{array}\right\}.$$
It is easily seen that the inclusion  $D_{0,k}^*\subseteq \dom(\mathbf{\widetilde L}_{0,k}^*)$ holds. Indeed, for any $\mathbf{U}=(u_j)_{j=1}^N\in D_{0,k}^*$ and $\bb V=(v_j)_{j=1}^N\in \dom(\mathbf{\widetilde L}_{0,k})$, denoting $\mathbf{U}^*=\opl^\alpha_k(\mathbf{U})\in L^2(\mathcal{G})$, we get
\begin{equation*}
\begin{split}
&(\mathbf{\widetilde L}_{0,k}\mathbf{V}, \mathbf{U})_2=(\opl^\alpha_k(\mathbf{V}),\bb U)_2\\&=(\bb V,\opl^\alpha_k(\mathbf{U}))_2+\sum\limits_{j=1}^N\left[-v'_ju_j+v_ju'_j\right]_0^\infty+\sum\limits_{j=1}^k\left[v'_ju_j-v_ju'_j\right]_{b_k-}^{b_k+}\\&=(\bb V,\opl^\alpha_k(\mathbf{U}))_2=(\bb V, \bb U^*)_2,
\end{split}
\end{equation*}
which, by definition of the adjoint operator, means that $\mathbf{U}\in \dom(\mathbf{\widetilde L}_{0,k}^*)$ or $D_{0,k}^*\subseteq \dom(\mathbf{\widetilde L}_{0,k}^*)$.

Let us show the inverse inclusion $D_{0,k}^*\supseteq \dom(\mathbf{\widetilde L}_{0,k}^*)$. Take $\bb U\in\dom(\mathbf{\widetilde L}_{0,k}^*)$, then for any $\bb V\in \dom(\mathbf{\widetilde L}_{0,k})$ we have
\begin{equation*}
\begin{split}
&(\mathbf{\widetilde L}_{0,k}\bb V, \bb U)_2=(\opl^\alpha_k(\bb V), \bb U)_2\\&=(\bb V, \opl^\alpha_k(\bb U))_2+\sum\limits_{j=1}^N\left[-v'_ju_j+v_ju'_j\right]_0^\infty+\sum\limits_{j=1}^k\left[v'_ju_j-v_ju'_j\right]_{b_k-}^{b_k+}\\&=(\bb V, \mathbf{\widetilde L}_{0,k}^*\bb U)_2=(\bb V, \opl^\alpha_k(\bb U))_2.
\end{split}
\end{equation*}
Thus, we arrive at the equality
\begin{equation}\label{adjoint'}
\begin{split}
  &\sum\limits_{j=1}^N\left[-v'_ju_j+v_ju'_j\right]_0^\infty+\sum\limits_{j=1}^k\left[v'_ju_j-v_ju'_j\right]_{b_k-}^{b_k+}\\&=\sum\limits_{j=1}^N v'_j(0){u_j(0)}+\sum\limits_{j=1}^kv'(b_k)({u_j(b_k+)}-{u_j(b_k-)})=0
\end{split}
\end{equation}
for any $\bb V\in \dom(\mathbf{\widetilde L}_{0,k})$.

$\bullet$ Let ${\bb W}=(w_j)_{j=1}^N\in \dom(\mathbf{\widetilde L}_{0,k})$ be such that $$w'_3(0)=...=w'_N(0)=w'_1(b_k)=...=w'_k(b_k)=0.$$
 Then   for  $\bb U\in \dom(\mathbf{\widetilde L}_{0,k}^*)$ from  \eqref{adjoint'} it follows that
\begin{equation}\label{adjoint1'}
\sum\limits_{j=1}^N w_{j} '(0)u_j(0)=w'_1(0){u_1(0)}+w'_2(0){u_2(0)}=0.
\end{equation}
Recalling  $\sum\limits_{j=1}^N  w_j'(0)= w_1'(0)+w_2'(0)=0$ and assuming $w_2'(0)\neq 0$, we obtain from \eqref{adjoint1'} the equality $u_1(0)=u_2(0)$.
 Repeating the similar arguments for ${\bb W}=({w}_j)_{j=1}^N\in \dom(\mathbf{\widetilde L}_{0,k})$ such that ${w}'_4(0)=...= {w}'_N(0)=w'_1(b_k)=...=w'_k(b_k)=0$, we  get $u_1(0)=u_2(0)=u_3(0)$ and so on. Finally taking ${\bb W}=({w}_j)_{j=1}^N\in \dom(\mathbf{\widetilde L}_{0,k})$ such that ${w}'_N(0)=w'_1(b_k)=...=w'_k(b_k)=0$, we  arrive at $u_1(0)=u_2(0)=...=u_{N-1}(0)$ and consequently $u_1(0)=u_2(0)=...=u_{N}(0)$.

$\bullet$  Let ${\bb W}=(w_j)_{j=1}^N\in \dom(\mathbf{\widetilde L}_{0,k})$ be  such that $w'_1(0)=...=w'_N(0)=w'_2(b_k)=...=w'_k(b_k)=0$, then from  \eqref{adjoint'} it follows 
\begin{equation*}\label{adjoint2}
\sum\limits_{j=1}^k w'_{j}(b_k)({u_j(b_k+)}-u_j(b_k-))=w'_1(b_k)({u_1(b_k+)}-{u_1(b_k-)})=0.
\end{equation*}
 Assuming  $w'_1(b_k)\neq 0$, we get $u_1(b_k+)=u_1(b_k-)$ or $u_1\in H^2(\mathbb{R}_+\setminus\{b_k\})\cap H^1(\mathbb{R}_+)$. Analogously we can show that  $u_j\in H^2(\mathbb{R}_+\setminus\{b_k\})\cap H^1(\mathbb{R}_+)$ for any $j\in\{1,...,k\}$.
 Thus, $\bb U\in D_{0,k}^*$ or $D_{0,k}^*\supseteq \dom(\mathbf{\widetilde L}_{0,k}^*)$ and \eqref{adjoint_graph'} holds.

Let us  show that the operator $\mathbf{\widetilde L}_{0,k}$ is non-negative. First, note that every component of the vector $\mathbf{V}=(v_j)_{j=1}^N\in H^2(\mathcal{G})$ satisfies the following identity
\begin{equation}\label{identity_graph_k}
-v_j''+\omega v_j-p(\varphi_{k,j})^{p-1}v_j=
\frac{-1}{\varphi_{k,j}'}\frac{d}{dx}\left[(\varphi_{k,j}')^2\frac{d}{dx}\left(\frac{v_j}{\varphi_{k,j}'}\right)\right],\,\, x\in\mathbb{R}_+\setminus \{b_k\}.
\end{equation}
Moreover, for $j\in\{k+1,...,N\}$ the above equality holds also for $b_k$ since $\varphi'_{k,j}(b_k)\neq 0,$  for $j\in\{k+1,...,N\}$.
Using the above equality and integrating by parts, we get for
 $\mathbf{V}\in \dom(\mathbf{\widetilde L}_{0,k})$
 \begin{equation*}\label{nonneJ_graph}\begin{split}
&(\mathbf{\widetilde L}_{0,k}\mathbf{V},\mathbf{V})_2=\sum\limits_{j=1}^k \Big( \int\limits_{0}^{b_k-} + \int\limits_{b_k+}^{+\infty}\Big)
(\varphi_{k,j} ')^2\left[\frac{d}{dx}\left(\frac{v_j}{\varphi_{k,j} '}\right)\right]^2dx\\&+
\sum\limits_{j=k+1}^N\int\limits^{\infty}_{0}(\varphi_{k,j}')^2\left[\frac{d}{dx}\left(\frac{v_j}{\varphi_{k,j} '}\right)\right]^2dx+
\sum\limits_{j=1}^N\left[-v_j'{v}_j+v_j^2\frac{\varphi_{k,j}''}{\varphi_{k,j}'}\right]^{\infty}_{0}\\&+\sum\limits_{j=1}^k\left[v_j'{v}_j-v_j^2\frac{\varphi_{k,j}''}{\varphi_{k,j}'}\right]^{b_k+}_{b_k-}=\sum\limits_{j=1}^k \Big( \int\limits_{0}^{b_k-} + \int\limits_{b_k+}^{+\infty}\Big)
(\varphi_{k,j} ')^2\left[\frac{d}{dx}\left(\frac{v_j}{\varphi_{k,j} '}\right)\right]^2dx\\&+
\sum\limits_{j=k+1}^N\int\limits^{\infty}_{0}(\varphi_{k,j}')^2\left[\frac{d}{dx}\left(\frac{v_j}{\varphi_{k,j} '}\right)\right]^2dx\geq 0.
\end{split}
\end{equation*}

The  equality $\sum\limits_{j=1}^k\left[v_j'{v}_j-v_j^2\frac{\varphi_{k,j}''}{\varphi_{k,j}'}\right]^{b_k+}_{b_k-}=0$ is due to the fact that  $b_k$ is a first-order  zero  for $\varphi_{k,j}'$ (i.e. $\varphi_{k,j}''(b_k)\neq 0$).
 Consider
 $$\bb L_{k}= \left(\Big(-\frac{d^2}{dx^2}\Big)\delta_{i,j}\right), \quad \dom({\bb L}_{k})=\dom(\mathbf{\widetilde L}_{0,k}).
 $$
 It is obvious that  $\dom(\mathbf{\widetilde L}_{0,k}^*)=\dom({\bb L}_{k}^*)$.
 Thus, due to  Neumann formula \eqref{d6}, we obtain the decomposition (when $\mathbf{\widetilde L}_{0,k}$ and $\bb L_k$ are assumed to act on complex-valued functions)
\begin{equation*}\label{Neumann}
\dom(\mathbf{\widetilde L}_{0,k}^*)=\dom(\mathbf{\widetilde L}_{0,k})\oplus\mathcal{N}_+(\widetilde{\bb L}_{0,k})\oplus\mathcal{N}_{-}(\widetilde{\bb L}_{0,k})=\dom(\mathbf{\widetilde L}_{0,k})\oplus\mathcal{N}_+({\bb L}_{k})\oplus\mathcal{N}_{-}({\bb L}_{k}),
\end{equation*}
where $\mathcal N_{\pm }({\bb L}_{k})=\Span\{\bb\Psi^0_{\pm i},\bb\Psi^1_{\pm i},..., \bb\Psi^k_{\pm i}\}$, with
 $\bb\Psi^0_{\pm i}=\left(e^{i\sqrt{\pm i}(x-b_k)}\right)_{j=1}^N$ and  $$\bb\Psi^m_{\pm i}=\left(\underset{\bb 1}{e^{i\sqrt{\pm i}(x-b_k)}},...,\underset{\bb m-1}{e^{i\sqrt{\pm i}(x-b_k)}},\underset{\bb m}{e^{i\sqrt{\pm i}(|x-b_k|-2b_k)}},\underset{\bb m+1}{e^{i\sqrt{\pm i}(x-b_k)}},..., \underset{\bb N}{e^{i\sqrt{\pm i}(x-b_k)}}\right),$$ where $m\in\{1,...,k\}.$ Note that $\Im(\sqrt{\pm i})$ is assumed to be positive.

Since $n_\pm (\bb L_k)=k+1$, by \cite[Chapter IV, Theorem 6]{Nai67}, it follows that $n_{\pm}(\mathbf{\widetilde L}_{0,k})=k+1$. Finally, due to  Proposition \ref{semibounded},  $n(\mathbf{L}^\alpha_{1,k})\leq k+1$.

$(ii)$ The proof is similar. In particular, we need to consider the operator $\bb L^\alpha_{1,k}$ as  the self-adjoint extension of the non-negative symmetric operator \begin{equation*}
\begin{split}
\mathbf{\widetilde L}_{0,N-k}=\opl^\alpha_k,\quad \dom(\mathbf{\widetilde L}_{0,N-k})=\left\{
\bb V\in D_\alpha: v_{k+1}(b_k)=...=v_N(b_k)=0 \right\},
\end{split}
\end{equation*}
where $b_k=-\frac{2}{(p-1)\sqrt{\omega}}a_k$.
The deficiency indices of $\mathbf{\widetilde L}_{0,N-k}$ equal $N-k$ (when $\mathbf{\widetilde L}_{0,N-k}$ is assumed to act on complex-valued functions). Indeed,  basically $\mathbf{\widetilde L}_{0,N-k}$ is the restriction of the operator $\bb L^\alpha_{1,k}$ onto the subspace of  codimension $N-k$. To show the non-negativity of $\mathbf{\widetilde L}_{0,N-k}$,  we need to use formula \eqref{identity_graph_k}. It induces
\begin{equation*}\label{nonneg_graph}\begin{split}
&(\mathbf{\widetilde L}_{0,N-k}\mathbf{V},\mathbf{V})_2=\sum\limits_{j=k+1}^N \Big( \int\limits_{0}^{b_k-}+\int\limits_{b_k+}^{+\infty}\Big)
(\varphi_{k,j} ')^2\left[\frac{d}{dx}\left(\frac{v_j}{\varphi_{k,j} '}\right)\right]^2dx\\&+
\sum\limits_{j=1}^k\int\limits^{\infty}_{0}(\varphi_{k,j}')^2\left[\frac{d}{dx}\left(\frac{v_j}{\varphi_{k,j} '}\right)\right]^2dx  +
\sum\limits_{j=1}^N\left[-v_j'{v}_j+v_j^2\frac{\varphi_{k,j}''}{\varphi_{k,j}'}\right]^{\infty}_{0}\\&+\sum\limits_{j=k+1}^N\left[v_j'{v}_j-v_j^2\frac{\varphi_{k,j}''}{\varphi_{k,j}'}\right]^{b_k+}_{b_k-}=\sum\limits_{j=k+1}^N \Big( \int\limits_{0}^{b_k-} + \int\limits_{b_k+}^{+\infty}\Big)
(\varphi_{k,j} ')^2\left[\frac{d}{dx}\left(\frac{v_j}{\varphi_{k,j} '}\right)\right]^2dx\\&+
\sum\limits_{j=1}^k\int\limits^{\infty}_{0}(\varphi_{k,j}')^2\left[\frac{d}{dx}\left(\frac{v_j}{\varphi_{k,j} '}\right)\right]^2dx+
\sum\limits_{j=1}^N\left[v_j'(0){v}_j(0)-v_j^2(0)\frac{\varphi_{k,j}''(0)}{\varphi_{k,j}'(0)}\right]\geq 0.
\end{split}
\end{equation*}
Indeed, $\sum\limits_{j=k+1}^N\left[v_j'{v}_j-v_j^2\frac{\varphi_{k,j}''}{\varphi_{k,j}'}\right]^{b_k+}_{b_k-}=0$ (see the proof of item $(i)$). Moreover,
 $$\sum\limits_{j=1}^N\left[v_j'(0){v}_j(0)-v_j^2(0)\frac{\varphi_{k,j}''(0)}{\varphi_{k,j}'(0)}\right]=\frac{v_1^2(0)(p-1)(\omega(N-2k)^2-\alpha^2)}{2\alpha}\geq 0.$$
 Finally,  due to Proposition \ref{semibounded}, we get  the result.
\end{proof}
\begin{remark}\label{Morse_est}
\begin{itemize}
\item[$(i)$] It is easily seen that
 $$\sum\limits_{j=1}^N\left[v_j'(0){v}_j(0)-v_j^2(0)\frac{\varphi_{k,j}''(0)}{\varphi_{k,j}'(0)}\right]=\frac{v_1^2(0)(p-1)(\omega(N-2k)^2-\alpha^2)}{2\alpha}\leq 0$$
  for $\alpha<0$, and therefore the  restriction  of  $\bb L^\alpha_{1,k}$  onto the subspace
 $$\{\bb V\in D_\alpha: v_1(b_k)=...=v_k(b_k)=0\}$$ is not  a non-negative operator. Thus, we need to assume additionally  that $v_1(0)=...=v_N(0)=0$.
  \item[$(ii)$]  The result of  item $(ii)$ (for $\alpha>0$) of the above Proposition can be extended to the case of $k=0$, i.e. $n(\bb L_{1,0}^\alpha)\leq N$. 
   \end{itemize}
\end{remark}

\end{document}